\documentclass[10pt,leqno]{amsart}
\usepackage{titlesec}
\titleformat{\subsection}
{\normalfont\bfseries}
{\thesubsection}
{1em}
{}
\titleformat{\section}
{\normalfont\Large\bfseries}
{\thesection\quad}
{0pt}
{}

\usepackage{indentfirst,csquotes}
\usepackage{listings}
\usepackage{mathtools}
\usepackage{adjustbox}
\usepackage{stmaryrd}
\usepackage{graphicx}
\usepackage{ulem}
\usepackage[most]{tcolorbox}
\usepackage{mdframed}
\usepackage{amssymb}
\usepackage{algorithm}
\usepackage{algpseudocode}
\usepackage{amsthm}
\usepackage{amsmath}
\usepackage{pbox}
\usepackage{xcolor}
\usepackage{hyperref}
\usepackage{etoolbox}
\usepackage{cleveref}

\input{preamble.sty}

\newtheorem{theorem}{Theorem}[section]

\newtheorem{lemma}[theorem]{Lemma}
\newtheorem{assumption}[theorem]{Assumption}
\newtheorem{problem}[theorem]{Problem}
\newtheorem{remark}[theorem]{Remark}

\newtheorem{corollary}[theorem]{Corollary}

\numberwithin{equation}{section}

\makeatletter
\renewcommand\paragraph[1]{\@startsection{paragraph}{4}{\parindent}%
{0.1ex \@plus 0.1ex \@minus 0.2ex}
{-1em}
{\normalfont\itshape\normalsize}*{#1.\quad}}
\makeatother

\hypersetup{
    colorlinks=true,
    linkcolor=blue,
    filecolor=blue,
    urlcolor=blue,
    citecolor=blue
}

\begin{document}
    \title[Multilevel Stochastic Gradient Descent]
    {\Large Multilevel Stochastic Gradient Descent \\ for Optimal Control Under Uncertainty}
    \author{Niklas Baumgarten}
    \author{David Schneiderhan}
    \date{\today}

    \maketitle

    \let\thefootnote\relax\footnotetext{\texttt{niklas.baumgarten@uni-heidelberg.de} \\ \indent \texttt{david.schneiderhan@kit.edu}}

    \maketitle

    \vspace{-1cm}

    \begin{abstract}
    We present a multilevel stochastic gradient descent method
    for the optimal control of systems governed
    by partial differential equations under uncertain input data.
    The gradient descent method used to find the optimal control leverages
    a parallel multilevel Monte Carlo method as stochastic gradient estimator.
    As a result, we achieve precise control over the stochastic gradient’s bias,
    introduced by numerical approximation, and its sampling error,
    arising from the use of incomplete gradients,
    while optimally managing computational resources.
    We show that the method exhibits linear convergence in the number
    of optimization steps while avoiding the cost of
    computing the full gradient at the highest fidelity.
    Numerical experiments demonstrate that the method significantly
    outperforms the standard (mini-)batched stochastic gradient
    descent method in terms of convergence speed and accuracy.
    The method is particularly well-suited for high-dimensional
    control problems, taking advantage of parallel computing resources
    and a distributed multilevel data structure.
    Additionally, we evaluate and implement different step size strategies,
    optimizer schemes, and budgeting techniques.
    The method’s performance is studied using a two-dimensional
    elliptic subsurface diffusion problem with log-normal
    coefficients and Matérn covariance.
\end{abstract}

    \section{Introduction}\label{sec:introduction}
The state of a physical, technological, or economical process,
as a function of space and time, is often described by partial differential equations (PDEs).
Controlling the state of such processes, for example,
by imposing boundary conditions or external forces,
is of significant interest in all the aforementioned applications.
However, determining the optimal control of a
PDE-governed system is a computationally demanding task,
particularly when the PDE involves high-dimensional, uncertain input data
and the control has to be found with high precision.

Known models to uncertain optimal control problems (OCPs)
focus on minimizing the expected distance between the state and a desired target.
Finding solutions to such problems often requires three key functionalities:
solving the minimization problem, e.g.~by using stochastic gradient descent (SGD) methods;
addressing the uncertainty through sufficient sampling of the input data;
and discretizing the PDEs with finite element (FE) methods.
Approaches which fall under this description can be found
in~\cite{geiersbach2019projected, geiersbach2023nonlinear, geiersbach2020stochastic, martin2021complexity, toraman2023momentum}.

Motivated by this work, we take an integrated approach, combining all three functionalities
in a single algorithm which leverages multilevel variance reduction,
as in multilevel Monte Carlo (MLMC) methods~\cite{cliffe2011multilevel, giles2015multilevel},
and parallel computing resources.
Even though other sampling methods to discretize the input space involve
sparse grids~\cite{kouri2014multilevel, nobile2024combination} and quasi-Monte Carlo
methods~\cite{guth2021quasi, guth2024parabolic},
we based our approach on the budgeted multilevel
Monte Carlo (BMLMC) method~\cite{baumgarten2023fully, baumgarten2024fully, baumgarten2025budgeted},
which provides high-performance and broad applicability.
As a baseline, we consider a step size controlled parallel (mini-)
batched stochastic gradient descent (BSGD) method,
e.g.~used in~\cite{chen2024minibatch, geiersbach2023optimization},
and show that our method improves it significantly in terms of convergence speed,
achievable accuracy, scalability and robustness
(cf.~\Cref{fig:metod-comparison} for a direct comparison).

As a model problem, we consider a two-dimensional elliptic subsurface
diffusion problem with log-normal coefficients (illustrations in~\Cref{fig:log-normal-fields-and-control}),
generated using the memory efficient stochastic PDE sampling technique of~\cite{kutri2024dirichlet}.
We intentionally choose a high-dimensional elliptic problem
to not lay the focus on the PDE, but rather on the algorithm and its properties.
To get the gradients for the SGD, BSGD or the here introduced
multilevel stochastic gradient descent (MLSGD) method,
we solve the adjoint system corresponding to the PDE constraint
as outlined in~\cite{geiersbach2019projected}, also related
to the adjoint Monte Carlo method described in~\cite{caflisch2024adjoint}.

Previous multilevel ideas for OCPs can be found
in~\cite{ali2017multilevel} for pathwise control,
or in~\cite{ciaramella2024multigrid} for sample average approximation (SAA)
assembled into a large linear system
— similar to~\cite{baumgarten2023fully} —
then solved using a multigrid algorithm.
The results in~\cite{guth2023multilevel, van2019robust} provide a foundation
for multilevel gradient estimation of OCPs, e.g.,
then used in a nonlinear conjugate gradient method.
The approaches in~\cite{geiersbach2020stochastic, martin2021complexity}
address the optimization problem using stochastic approximation (SA),
tracing back to~\cite{robbins1951stochastic}, in the form of a SGD method.
Our method embeds the mentioned multilevel gradient estimation in this SGD approach.
In~\cite{noufel2016multilevelstochasticapproximation},
the SA is also extended to a multilevel setting,
considering~\cite{giles2008multilevel},
and a central limit theorem is shown
for a multilevel iteration scheme similar to the one we propose.
From this we draw further motivation to extend the
application from stochastic differential
equations as in~\cite{noufel2016multilevelstochasticapproximation}
to PDE constrained optimization under uncertainty,
and to develop a scalable and adaptive algorithm.
Related ideas can also be found in the context of Bayesian inverse problems,
where the Stein variation gradient descent has also been
extended by multilevel ideas~\cite{weissmann2022multilevel,weissmann2024mean}.
We point out that~\cite{nobile2025multilevel} follows along those lines, too.
Here, the objective functional is decomposed with a telescoping sum
and the results of optimization problems on different levels are determined and combined.
Our approach decomposes the gradient estimation into multiple
Monte Carlo estimators and
includes similar adaptive sampling
strategies as~\cite{beiser2023adaptive}.
Running the algorithm has then the advantage that
little a priori knowledge on the problem
is required by leveraging collected multilevel data
to guide the adaptivity.

Machine learning has been the main driver for the development of novel SGD methods
to enable large-scale training of neural networks,
often striking the balance between per-iteration cost
and expected improvements~\cite{bottou2018optimization}.
Features such as variance (noise) reduction and parallelism through
BSGD~\cite{khirirat2017mini}, adaptive moment estimation (ADAM)~\cite{kingma2014adam},
adaptive step sizes~\cite{koehne2024adaptivestepsizespreconditioned},
or averaging and aggregation schemes~\cite{polyak1992averaging, shamir2013optimalAveragingSGD},
have been successfully applied to various problems.
We take inspiration in these approaches, incorporating them into the proposed MLSGD method for OCPs
and conjecture that the algorithm is applicable in machine learning as well.

In conclusion, recent developments, theory and existing algorithms motivate
the combination of multilevel variance reduction and SGD methods,
however, we have not found an adaptive and parallel method
for high-dimensional control problems yet.
With this paper, we propose new ways to realize such a method
which in particular features:

\paragraph{Algorithmic description and convergence analysis}
Building on the assumptions and notations in~\Cref{sec:assumptions-and-notation},
we present a detailed description of the MLSGD method for OCPs
in~\Cref{alg:bsgd} and~\Cref{alg:mlsgd},
incorporating the standard BSGD method recalled in~\Cref{sec:batched-stochastic-gradient-descent}
and the MLMC estimation from~\Cref{sec:multilevel-monte-carlo}.
In~\Cref{sec:multilevel-stochastic-gradient-descent},
we establish the linear convergence of the proposed method in terms of optimization steps
through \Cref{thm:convergence-mlsgd} and \Cref{cor:convergence-mlsgd}.
This result follows from bounding the error of the gradient estimation through an appropriate
choice of the multilevel batch size (see~\Cref{lem:error-gradient-estimation}).
The method achieves this convergence rate (like an SAA approach)
with a complexity similar to the standard MLMC method,
as it circumvents full gradient evaluations (as done for SA) at the highest level.
As a result, MLSGD improves convergence rates, speed, and accuracy
by adapting the multilevel batches.

\paragraph{Adaptivity, budgeting and error control by resources}
As an extension to the new MLSGD method presented in~\Cref{sec:multilevel-stochastic-gradient-descent},
we incorporate the adaptive step size rule from~\cite{koehne2024adaptivestepsizespreconditioned},
adaptive multilevel batches similar to~\cite{van2019robust},
the optimal distribution of the computational load as in~\cite{baumgarten2025budgeted},
and a posteriori error control with dynamic
programming~\cite{baumgarten2024fully} into~\Cref{alg:bmlsgd}.
We impose the given computational resources,
such as the total available memory and the reserved CPU-time budget,
as additional constraints to the optimization problem resulting in~\Cref{problem:knapsack-ocp}.
The final Budgeted Multilevel Stochastic Gradient Descent (BMLSGD) method
allows for total error control by the given computational resources
through~\Cref{cor:upper-and-lower-bound} and is presented in~\Cref{alg:bmlsgd}.

\paragraph{Numerical experiments with HPC resources}
Lastly, we note that the method is designed for High-Performance Computing (HPC) resources
implemented on a domain- and sample-distributed multiindex data
structure~\cite{baumgarten2025budgeted},
enabling efficient use of resources as in~\cite{baumgarten2024fully}.
While our study focuses on an elliptic model problem,
the algorithm and implementation are designed for broader applicability,
demonstrating parallel scalability through the usage of the FE software
M++~\cite{baumgarten2021parallel}, here used in version~\cite{wieners2025mpp350}.
In our numerical experiments, presented across several sections
(see Sections~\ref{subsec:bsgd-experiments},~\ref{subsec:mlsgd-experiments}
and~\ref{subsec:bmlsgd-experiments}),
we compare all introduced algorithms
(see~Sections~\ref{subsec:bsgd-algorithm}, \ref{subsec:mlsgd-algorithm},
\ref{subsec:adaptivity-and-budgeting-algorithm})
to support the theoretical findings mentioned in the previous paragraphs.
We particularly highlight again~\Cref{fig:metod-comparison} illustrating
the superior performance of BMLSGD over BSGD, and~\Cref{fig:nodes}
which demonstrates that this method also scales well with
increased computational resources.

    \section{Assumptions and Notations}\label{sec:assumptions-and-notation}

The goal of the proposed methodology is to find the optimal control
to a PDE governed system with uncertain input data.
Similar problems are considered for example
in~\cite{geiersbach2020stochastic, guth2023multilevel, martin2021complexity},
presenting theoretical foundation for our approach;
however, the algorithm and the notation are closely related
to~\cite{baumgarten2023fully, baumgarten2024fully, baumgarten2025budgeted}.
The PDE system of interest is defined on a bounded polygonal spatial domain $\cD \subset \RR^d$ with $d \in \set{1, 2, 3}$,
while the uncertainty of the input data is captured by a complete probability space $(\Omega, \cF, \PP)$.
Let further $(V, \sprod{\cdot, \cdot}_V), (W, \sprod{\cdot, \cdot}_W)$ denote Hilbert spaces with
$V\subseteq W \subseteq \rL^2(\cD)$,
where $V$ is an appropriate space for an PDE sample solution.
Lastly, let $\rL^2(\Omega, V)$ and $\rL^2(\Omega, W)$ denote Bochner spaces
containing all $\rL^2$-integrable maps from the probability space
$(\Omega, \cF, \PP)$ to $V$ and $W$, respectively.

    \subsection{Optimization Problem}\label{subsec:optimization-problem}

We search for an admissible control $\bz \in Z$ to an optimal control problem (OCP),
where $Z$ is a non-empty, closed and convex set
\begin{align*}
    Z \coloneqq \set{\bz \in W \colon {\bz_{\mathrm{ad}}^{\mathrm{low}}(\bx)}\leq\bz(\bx)
    \leq {\bz_{\mathrm{ad}}^{\mathrm{up}}(\bx)}, \text{ a.e. }
    \bx \in \cD}
\end{align*}
with $\bz_{\mathrm{ad}}^{\mathrm{low}}, \bz_{\mathrm{ad}}^{\mathrm{up}}\in W$.
The control is found if the distance of some prescribed
target $\bd~\in~W$ to the state solution $\bu \in \rL^2(\Omega, V)$
of the PDE is minimal in expectation.

\begin{problem}[Optimal Control under Uncertainty]
    \label{problem:ocp}
    Given the desired target state $\bd \in W$ and a cost factor $\lambda\geq 0$, find the optimal,
    admissible and deterministic control $\bz \in Z$, such that
    \begin{equation}
        \label{eq:ocp-objective}
        \min_{\bz \in Z} J(\bz)
        \,\,\, \text{with} \,\,\,
        J(\bz) \coloneqq \EE \squarelr{j(\cdot, \bz)}
        \,\,\, \text{and} \,\,\,
        j(\omega, \bz) \coloneqq
        \tfrac{1}{2} \norm{\bu[\omega] - \bd}^2_{W}
        + \tfrac{\lambda}{2} \norm{\bz}_{W}^2
    \end{equation}
    under the constraint that the state $\bu \in \rL^2(\Omega, V)$ is the solution of
    \begin{equation}
        \label{eq:ocp-constraint}
        \cG[\omega] \, \bu(\omega, \bx) = \bz(\bx)
    \end{equation}
    with $\cG[\omega]$ representing a linear uncertain PDE system.
\end{problem}

To ensure a unique solution to~\Cref{problem:ocp}
and to show convergence of the used SGD methods,
we suppose the following conditions on the optimization problem.

\clearpage

\begin{assumption}
    \label{assumption:solvable}
    \, \\[-4mm]
    \begin{enumerate}
        \item The functional $j \colon \Omega \times W \to \RR$ is $\rL^2$-Fréchet differentiable on $W$,
        i.e.,~for every open $\cZ \subset W$ containing $\bz$, there exists a linear operator
        $\cA \colon \Omega \times \cZ \to \cL(W, \RR)$, such that
        \begin{equation*}
            \lim_{W \ni \bh \to 0} \frac{\norm{j(\cdot, \bz + \bh) - j(\cdot, \bz) + \cA(\cdot, \bz)\bh }_{\rL^2(\Omega)}}{\norm{\bh}_W}=0.
        \end{equation*}

        \smallskip

        \item The mapping $\bz \mapsto j(\omega, \bz)$ is strongly $\mu$-convex on
        the admissible set $Z$ for almost every $\omega \in \Omega$,
        implying $\bz \mapsto J(\bz)$ being strongly $\mu$-convex,
        too, i.e., there exists a constant $\mu > 0$,
        such that for all $\bz^{(1)}\!, \bz^{(2)} \in Z$,
        \begin{align}
            \label{eq:strongly-convex}
            \bsprod{\nabla J[\bz^{(1)}]  -  \nabla J[\bz^{(2)}], \, \bz^{(1)}  - \bz^{(2)}}_W
            \geq \mu \bnorm{\bz^{(1)}  - \bz^{(2)}}_W^2 \quad \text{and} \,\,\, \\
            \label{eq:strongly-convex-equivalent}
            J(\bz^{(2)}) \geq J(\bz^{(1)}) + \bsprod{\nabla J[\bz^{(1)}], \, \bz^{(2)}  - \bz^{(1)}}_W
            + \tfrac{\mu}{2} \bnorm{\bz^{(2)} - \bz^{(1)}}_W^2.
        \end{align}

        \smallskip

        \item $\nabla J$ is Lipschitz continuous, i.e.,
        there exists a constant $c_{\mathrm{Lip}} > 0$, such that
        \begin{equation}
            \label{eq:lipschitz-continuity-gradient}
            \bnorm{\nabla J[\bz^{(1)}] - \nabla J[\bz^{(2)}]}_{W} \leq c_{\mathrm{Lip}} \bnorm{\bz^{(1)} - \bz^{(2)}}_W,
            \quad \forall \, \bz^{(1)}\!, \bz^{(2)} \in Z.
        \end{equation}
    \end{enumerate}
\end{assumption}

Under~\Cref{assumption:solvable}, the solution $\bz^* \in Z$ to~\Cref{problem:ocp}
satisfies besides~\eqref{eq:ocp-constraint} also the
variational inequality (cf. \cite{hinze2009optimization})
\begin{equation}
     \label{eq:variational-inequality}
     \bsprod{{\nabla J[\bz^{*}]} , \bz - \bz^*}_W \geq 0 \quad \forall \bz \in Z
\end{equation}
and if $\bz^*\in Z$ with
$\bz_{\mathrm{ad}}^{\mathrm{low}}(\bx) < \bz(\bx) < \bz_{\mathrm{ad}}^{\mathrm{up}}(\bx)$
for a.e.~$\bx \in \cD$, we even get
\begin{equation}
    \label{eq:optimality-condition}
    {\nabla J[\bz^{*}](\bx)}=0 \quad \text{with} \quad \nabla J[\bz] (\bx) = \lambda \bz(\bx) - \EE[\bq](\bx),
\end{equation}
where $\bq \in \rL^2(\Omega, V)$ is the solution of the adjoint PDE system
\begin{equation}
    \label{eq:adjoint-pde}
    \cG^*[\omega] \, \bq(\omega, \bx) = \bd(\bx) - \bu(\omega, \bx)
\end{equation}
and $\cG^*[\omega]$ is the uncertain adjoint system of $\cG[\omega]$.
Note that the equivalent
assumptions~\eqref{eq:strongly-convex} and~\eqref{eq:strongly-convex-equivalent}
are already implied if $\lambda > 0$ and
the operator $\bz \mapsto \bu(\omega)$ is linear.
    \subsection{Approximation}\label{subsec:approximation-problem}

We want to find approximate solutions to~\Cref{problem:ocp},
which involves three main components:
(i)
computing finite element (FE) solutions $\bu_\ell \in V_\ell$
of system~\eqref{eq:ocp-constraint} and $\bq_\ell \in V_\ell$ of
system~\eqref{eq:adjoint-pde} for a fixed $\omega \in \Omega$,
$V_\ell$ denoting a suitable finite element space associated with $V$ at discretization level $\ell$;
(ii)
sampling finite-dimensional representations of stochastic events
$\Omega \ni \omega \mapsto \by_\ell \in V_\ell$ to generate the input data for the PDE systems;
(out of simplicity, we take $V_\ell$ for all functions with a discrete representation,
but remark that different spaces can be chosen as well.
For an illustration of some $\by_\ell \in V_\ell$,
we refer to~\Cref{fig:log-normal-fields-and-control}) and
(iii)
finding the solution of~\eqref{eq:optimality-condition}
using an iterative stochastic approximation indexed by $k \in \NN_0$.
The iteration scheme of a standard SGD method is given
with step sizes $t_k > 0$ by
\begin{equation}
    \label{eq:sgd-iteration}
    \bz^{(k+1)}_\ell \!\leftarrow\! \pi_{Z} \big(\bz_{\ell}^{(k)}
    \!- t_k (\lambda \bz_{\ell}^{(k)} \!- \bq_{\ell}^{(k)}) \big),
    \quad \text{where} \quad
    \pi_{Z} (\bz_\ell) = \argmin_{\bw \in Z} \norm{\bz_\ell - \bw}_{W}
    \vspace{-2mm}
\end{equation}
ensures that the control $\bz_{\ell}^{(k+1)}$ remains in
the admissible discretized space.
This approach, using $\bg_\ell^{(k)} = \lambda \bz_{\ell}^{(k)} \!- \bq_{\ell}^{(k)}$
as stochastic gradient, is motivated by~\cite{geiersbach2019projected},
with the methodology further developed in~\cite{geiersbach2020stochastic},
where the step size $t_k$ is assumed to satisfy
\begin{equation}
    \label{eq:step-size-assumption}
    t_k > 0, \quad \sum_{k=1}^{\infty} t_k = \infty, \quad \sum_{k=1}^{\infty} t_k^2 < \infty,
\end{equation}
to ensure convergence
(cf.~\cite{bottou2018optimization, koehne2024adaptivestepsizespreconditioned}
for further reading on step size control).
This is combined with an adaptive mesh refinement and a step size decay
to control the sampling error and the bias introduced by the FE method.
Here, we approach this by introducing a MLMC estimator for the gradients,
controlling both errors while drastically speeding up the computations
and reducing the variance to make the choice of the step sizes less critical.

To outline all necessary assumptions for this approach,
we consider an increasing sequence of sub-$\sigma$-algebras
$\tset{\cF_k}_{k\in \NN_0}$ of $\cF$ (a filtration),
such that $\bz_\ell^{(k)}$ and $\bz^{(k)}$ are $\cF_k$-measurable.
Since all quantities updated in the optimization depend upon the previous steps,
we introduce with the conditional expectation $\EE_k[\, \cdot \, ]~\coloneqq~\EE[\, \cdot \, | \cF_k]$
(cf. \cite{lord2014computationalSPDEs, pflug1996optimization})
the positive random variable
$\tnorm{\bv^{(k)}}_{\rL^2_k(\Omega, W)}^2 \coloneqq \EE_k \big[\tnorm{\bv^{(k)}}_{W}^2 \big]$
inheriting the properties of the $W$ norm and the space
\begin{align*}
    \rL^2_k(\Omega, W) \coloneqq
    \set{\bv^{(k)} \in \rL^2(\Omega, W) \, \colon \,
    \tnorm{\bv^{(k)}}_{\rL_{{k}}^2(\Omega, W)}^2 < \infty}.
\end{align*}
We remark that the computed control in~\eqref{eq:sgd-iteration}
is a discretized random field
$\bz_{\ell}^{(k)}\in\rL^2_k(\Omega, Z)$ unlike
the true solution $\bz^* \in Z$ to~\Cref{problem:ocp}, which is deterministic.

As a gradient estimator, the MLMC method gives its computational advantage
through the variance reduction of the state and adjoint level differences
\begin{align}
    \label{eq:definition-vell}
    \bv_\ell^{(k)} \! \coloneqq \bu_\ell^{(k)} \! - \rP_{\ell - 1}^{\ell} \bu_{\ell - 1}^{(k)},
    \, \bv_0^{(k)} \! \coloneqq \bu_0^{(k)}
    \text{ and }
    \bp_\ell^{(k)} \! \coloneqq \bq_{\ell}^{(k)} \! - \rP_{\ell - 1}^{\ell} \bq_{\ell-1}^{(k)},
    \, \bp_0^{(k)} \! \coloneqq \bq_{0}^{(k)} \!,
\end{align}
where we made use of a linear isometric projection operator
$\rP_{\ell - 1}^{\ell} \colon V_{\ell - 1} \rightarrow V_{\ell}$
by evaluating the coarser level FE function at the nodes of the finer level
for linear FE spaces.
The assumptions required for the MLMC estimation are summarized below.

\begin{assumption}
    \label{assumption:mlmc}
    The FE approximations of~\eqref{eq:ocp-constraint} and~\eqref{eq:adjoint-pde}
    with mesh diameter $h_\ell = h_0 \, 2^{-\ell}$ satisfy in every optimization step $k$

    \vspace{-2mm}

    \hspace{-7mm}
    \begin{minipage}{0.49\textwidth}
        \begin{subequations}
            \label{eq:assumptions-random-field}
            \begin{align}
                \label{eq:assumption-alpha-u}
                \bnorm{\EE_k[{\bu}_\ell^{(k)} - \bu^{(k)}]}_W \quad\, &\leq u_k h_\ell^{\alpha_{\bu}} \\
                \label{eq:assumption-beta-u}
                \bnorm{\bv_\ell^{(k)} \!\!- \!\EE_k[\bv_\ell^{(k)}]}^2_{\rL_k^2(\Omega, W)} \!\! &\leq v_k h_\ell^{\beta_{\bv}}
            \end{align}
        \end{subequations}
    \end{minipage}
    \begin{minipage}{0.49\textwidth}
        \begin{subequations}
            \label{eq:assumptions-qoi}
            \begin{align}
                \label{eq:assumption-alpha-q}
                \bnorm{\EE_k \big[{\bq}_\ell^{(k)} - \bq^{(k)} \big]}_W \quad\, &\leq q_k h_\ell^{\alpha_{\bq}} \\
                \label{eq:assumption-beta-q}
                \bnorm{\bp_\ell^{(k)} \!\! - \! \EE_k[\bp_\ell^{(k)}]}^2_{\rL_k^2(\Omega, W)} \!\! &\leq p_k h_\ell^{\beta_{\bp}}
            \end{align}
        \end{subequations}
    \end{minipage}

    \vspace{2mm}

    \noindent
    with exponents $\alpha_{\bu}, \alpha_{\bq}, \beta_{\bv}, \beta_{\bp} > 0$
    and $u_k, q_k, v_k, p_k > 0$
    being independent of the discretization level $\ell$.
    The state $\bu^{(k)}$ and the adjoint $\bq^{(k)}$  represent the true solutions to
    the continuous problems~\eqref{eq:ocp-constraint} and~\eqref{eq:adjoint-pde}
    given some control $\bz^{(k)}$.
    We further assume that the computation of the state-adjoint pair $(\bv_\ell^{(k)}, \bp_\ell^{(k)})$
    is bounded with $\gamma_{\rC\rT}, \gamma_{\mathrm{Mem}} > 0$
    and $c_k, m_k > 0$ in its computing-time and memory footprint

    \vspace{-2mm}
    \hspace{-7mm}
    \begin{minipage}{0.48\textwidth}
        \begin{equation}
            \label{eq:assumption-gamma-CT}
            \rC^{\rC\rT} \big(({\bv}_\ell^{(k)}, \bp_\ell^{(k)}) \big) \,\,
            \leq c_k h_\ell^{-\gamma_{\rC\rT}}
        \end{equation}
    \end{minipage}
    \begin{minipage}{0.51\textwidth}
        \begin{equation}
            \label{eq:assumption-gamma-mem}
            \rC^{\mathrm{Mem}} \big((\bv_\ell^{(k)}, \bp_\ell^{(k)})\big)
            \leq m_k h_\ell^{-\gamma_{\mathrm{Mem}}}.
        \end{equation}
    \end{minipage}

    \vspace{0.2cm}

    \noindent Lastly, with $c_\cG, z_k>0$ and $\alpha_{\bz} > 0$ being also independent of $\ell$,
    we suppose
    \begin{equation}
        \label{eq:assumption-alpha-z}
        \bnorm{\bz_\ell^{(k)} - \bz^{(k)}}_{\rL_k^2(\Omega, W)}^2
        \leq c_{\cG} \EE_k \big[\bnorm{\bq_\ell^{(k)} - \bq^{(k)}}_W^2 \big]
        \leq z_k h_{\ell}^{2 \alpha_{\bz}}.
    \end{equation}
\end{assumption}

\begin{remark}
    \label{remark:assumptions}
    \begin{enumerate}
        \item
        \Cref{assumption:mlmc} is based
        on~\cite{baumgarten2025budgeted} and adapted
        to fit to the MLSGD method by extending it to
        every optimization step.

        \item
        The constants $u_k, q_k, v_k, p_k, c_k, m_k$ and $z_k$
        possess $k$-dependence due to changing right-hand sides of~\eqref{eq:ocp-constraint}
        and~\eqref{eq:adjoint-pde} as the optimization runs
        (cf.~arguments of~\cite{barth2011multi, charrier2013finite, teckentrup2013further} for an elliptic PDE).

        \item
        The exponents $\alpha_{\bu}, \alpha_{\bq}, \beta_{\bv}, \beta_{\bp}, \gamma_{\rC\rT}, \gamma_{\mathrm{Mem}}$
        and $\alpha_{\bz}$ are assumed to be independent of $k$ as the regularity
        of the PDE solutions is not expected to change during the optimization.

        \item 
        In the setting of this paper we denote
        $\alpha \coloneq \alpha_\bq, \beta \coloneq \beta_\bp$
        and $\gamma \in \{\gamma_{\rC\rT}, \gamma_{\mathrm{Mem}}\}$.

        \item
        Though assumptions~\eqref{eq:assumption-alpha-u} and~\eqref{eq:assumption-beta-u}
        are not used in theory, we present numerical estimates on them in
        \Cref{subsec:bmlsgd-experiments}

        \item
        The cost bounds~\eqref{eq:assumption-gamma-CT} and~\eqref{eq:assumption-gamma-mem},
        while in~\cite{baumgarten2025budgeted} imposed for full field MLMC estimation,
        also might find justification in machine learning
        through a formulation with respect to the problem size.
        Commonly, computing time budgets and memory constraints
        are among the limiting factors in the training of neural networks.

        \item
        The inequality on the control error~\eqref{eq:assumption-alpha-z}
        is used in~\Cref{lem:error-gradient-estimation}
        to avoid specific assumptions on the
        operators $\cG[\omega]$ and $\cG^*[\omega]$
        as well as the applied discretization schemes.
        This is motivated by
        inserting~\eqref{eq:ocp-constraint} and~\eqref{eq:adjoint-pde}
        into the left-hand side of~\eqref{eq:assumption-alpha-z} and
        thereof, $c_{\cG}$ encodes information about the operators.
        We note that this includes the assumption that $c_{\cG}$
        is independent of the optimization step $k$.

        \item
        By the work
        in~\cite{baumgarten2024fully, baumgarten2025budgeted, collier2015continuation, giles2015multilevel}
        the asymptotic behaviour of the inequalities in~\Cref{assumption:mlmc} can
        be estimated during runtime of the algorithm,
        enabling adaptivity and verification of the assumptions
        as done in~\Cref{subsec:bmlsgd-experiments}.
        Note that for~\eqref{eq:assumption-alpha-z}, we only
        measure the right-hand side indicated as gray dots.
    \end{enumerate}
\end{remark}

    \subsection{Example PDE}\label{subsec:example-problem}

As an example of the system $\cG[\omega]$, we consider the 
elliptic diffusion equation on $W=\rL^2(\cD)$ and $V=\rH^1_0(\cD)$ with log-normal coefficients
and homogeneous Dirichlet boundary conditions
\begin{equation}
    \pdeProblem{
        -\div\big(\exp(\by(\omega, \bx)) \nabla \bu(\omega, \bx) \,\big) &=& \bz(\bx) &\text{on } \,\, \Omega\times\cD \\
        \bu(\omega, \bx) &=& 0 &\text{on } \,\,  \Omega\times \partial \cD \,.
    }\label{eq:elliptic-pde}
\end{equation}
The adjoint operator $\cG^*[\omega]$ of this system is the same as $\cG[\omega]$ due to the problem's symmetry.
This problem will serve as an example for the numerical experiments in this paper
(cf.~\cite{geiersbach2020stochastic, guth2023multilevel} for more analysis);
however, we emphasize that the methodology is not limited to this specific PDE
as we do not rely on any properties other than already stated in the
Sections~\ref{subsec:optimization-problem} and~\ref{subsec:approximation-problem}.
The majority of the functionality used within the presented algorithms has already been used for
hyperbolic PDE systems~\cite{baumgarten2025budgeted, baumgarten2024fully, baumgarten2021parallel}.
Further analysis on other PDEs is also given, for example,
in~\cite{geiersbach2023optimization, guth2024parabolic, toraman2023momentum}.

For our numerical experiments, we consider the unit square $\cD = (0, 1)^2$ as domain,
impose the target state $\bd(\bx) =  \sin(2 \pi x_1) \sin(2 \pi x_2)$,
consider the cost factor $\lambda=10^{-8}$
and set the admissible bounds to $\bz_{\mathrm{ad}}^{\mathrm{low}} \equiv -1000, \bz_{\mathrm{ad}}^{\mathrm{up}}\equiv 1000$.
This leads to $\pi_Z$ having no impact on the presented experiments,
but can be adapted to the setting of~\cite{geiersbach2020stochastic}.
To define the log-normal diffusion coefficient in~\eqref{eq:elliptic-pde},
we consider a Gaussian random field (GRF), denoted $\by(\omega, \bx)$, with
mean-zero and the Matérn covariance function
\begin{align*}
    \Cov(\bx_1, \bx_2) = \frac{\sigma^2}{2^{\nu - 1} \Gamma(\nu)} (\kappa \br)^{\nu} \cK_{\nu}(\kappa \br), \quad
    \br = \norm{\bx_1 - \bx_2}_2, \quad \kappa = \frac{\sqrt{2\nu}}{\lambda_\kappa},
\end{align*}
where $\Gamma$ is the gamma function and $\cK_\nu$ the modified Bessel function of the second kind.
Realizations of the GRFs are computed with SPDE sampling~\cite{lindgren2011explicit},
particularly with the method introduced in~\cite{kutri2024dirichlet}.
Throughout the experiments, we set the parameters
$\sigma^2 = 1.5$, $\nu = 1$ and $\lambda_\kappa = 0.1$.
For two independent realizations $\by_\ell^{(m)}$ of the GRF with different mesh diameters,
we refer to the four plots on the left of \Cref{fig:log-normal-fields-and-control}.
The plots on the right show the computed control $\bz^{(k)}_\ell$ after $k=10$ and $k=100$
iterations of~\eqref{eq:bsgd-iteration}.

We solve all PDE systems
(needed for the Dirichlet-Neumann averaging of~\cite{kutri2024dirichlet}
as well as the state and adjoint equations)
with standard Lagrange linear FE,
geometric V-cycle Jacobi multigrid preconditioning
and CG methods.
We note that a large selection of other solvers and FE spaces,
as explored in~\cite{baumgarten2023fully}, are available and applicable
in the used software~\cite{baumgarten2021parallel} and within the proposed algorithms.

We remark that most numerical experiments are performed
on a single node on the HoreKa supercomputer utilizing $64$ CPUs.
In \Cref{subsec:bmlsgd-experiments}, we present scaling experiments
of moderate size up to $16$ nodes with a total of $1024$ CPUs.
Though the presented numerical results are for two-dimensional domains $\cD$,
the method is designed with three-dimensional domains in mind.
First numerical results for this were already achieved, however,
omitted as this also requires an in-depth discussion of the HPC
techniques and the memory layout which is,
as well as the PDE system,
not the main focus of this paper.

\begin{figure}
    \begin{center}
        \raisebox{0.55in}{\rotatebox{90}{$h_\ell=2^{-7}$}}
        \includegraphics[trim={7cm 1.8cm 7cm 0}, clip, width=0.25\textwidth]{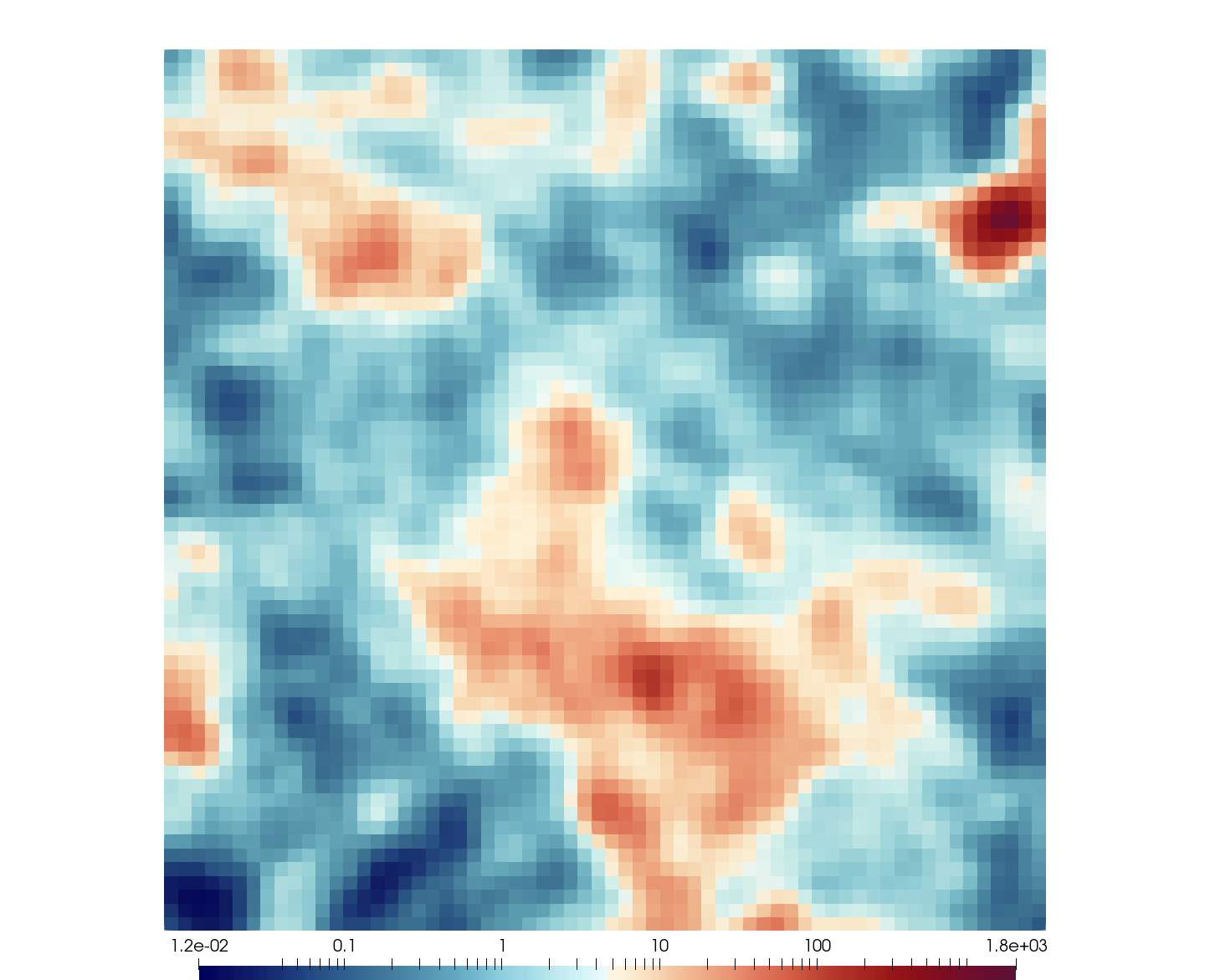}
        \includegraphics[trim={7cm 1.8cm 7cm 0}, clip, width=0.25\textwidth]{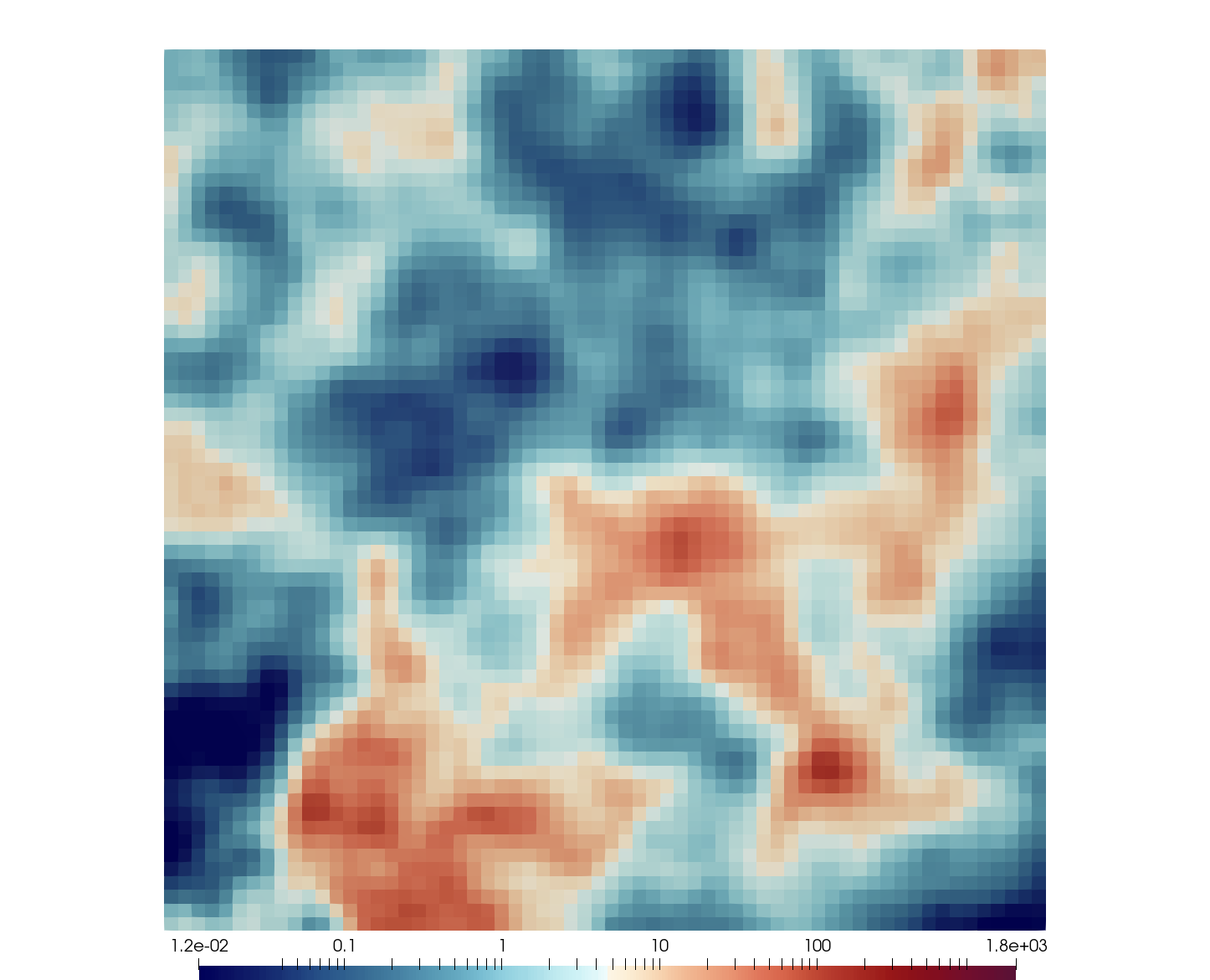}
        \hspace{0.1mm}
        \includegraphics[trim={7cm 1.8cm 7cm 0}, clip, width=0.25\textwidth]{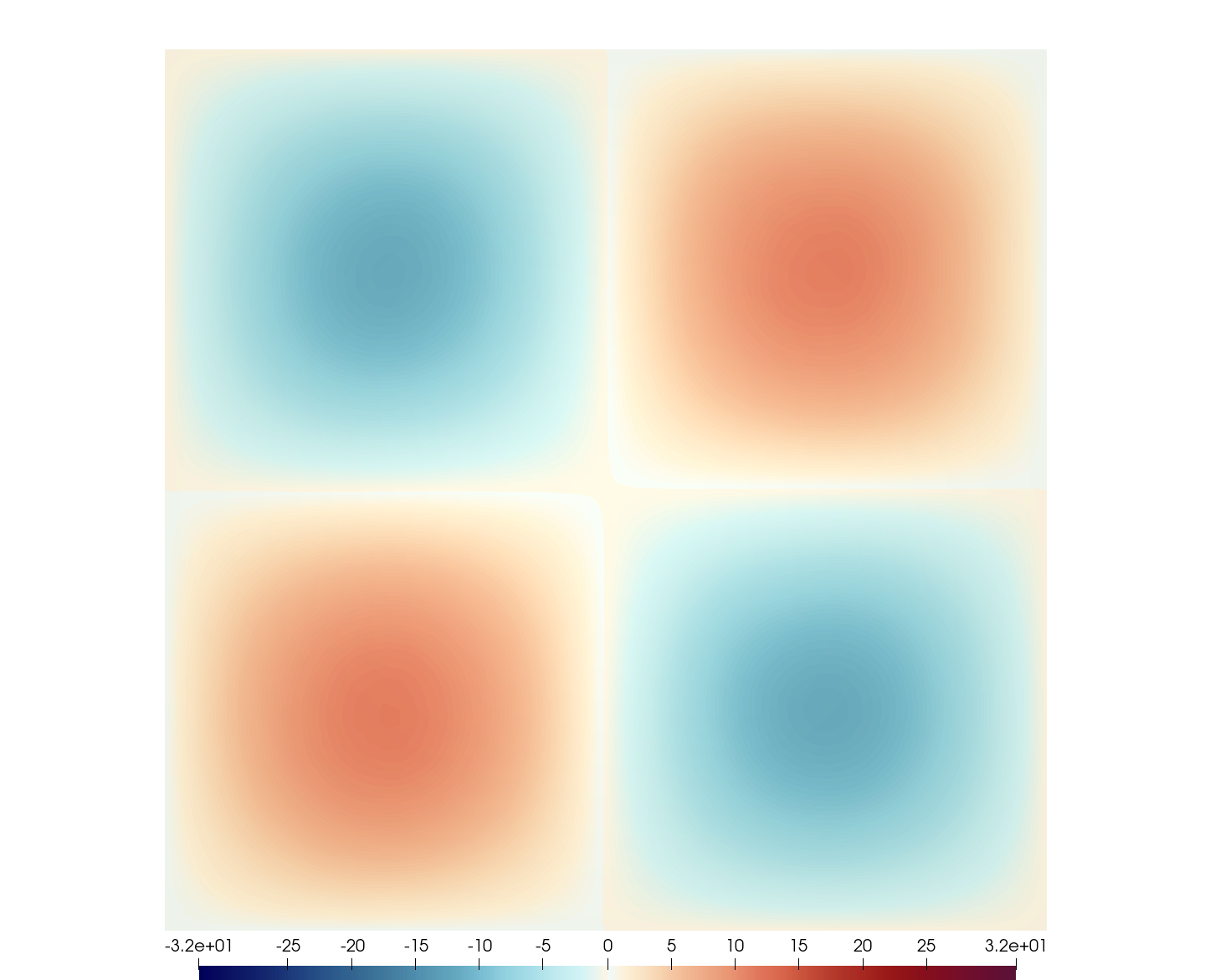}
        \raisebox{1.0in}{\rotatebox{-90}{$k=10$}}

        \vspace*{-1.5mm}

        \raisebox{0.55in}{\rotatebox{90}{$h_\ell=2^{-8}$}}
        \includegraphics[trim={7cm 2cm 7cm 0.4cm}, clip, width=0.25\textwidth]{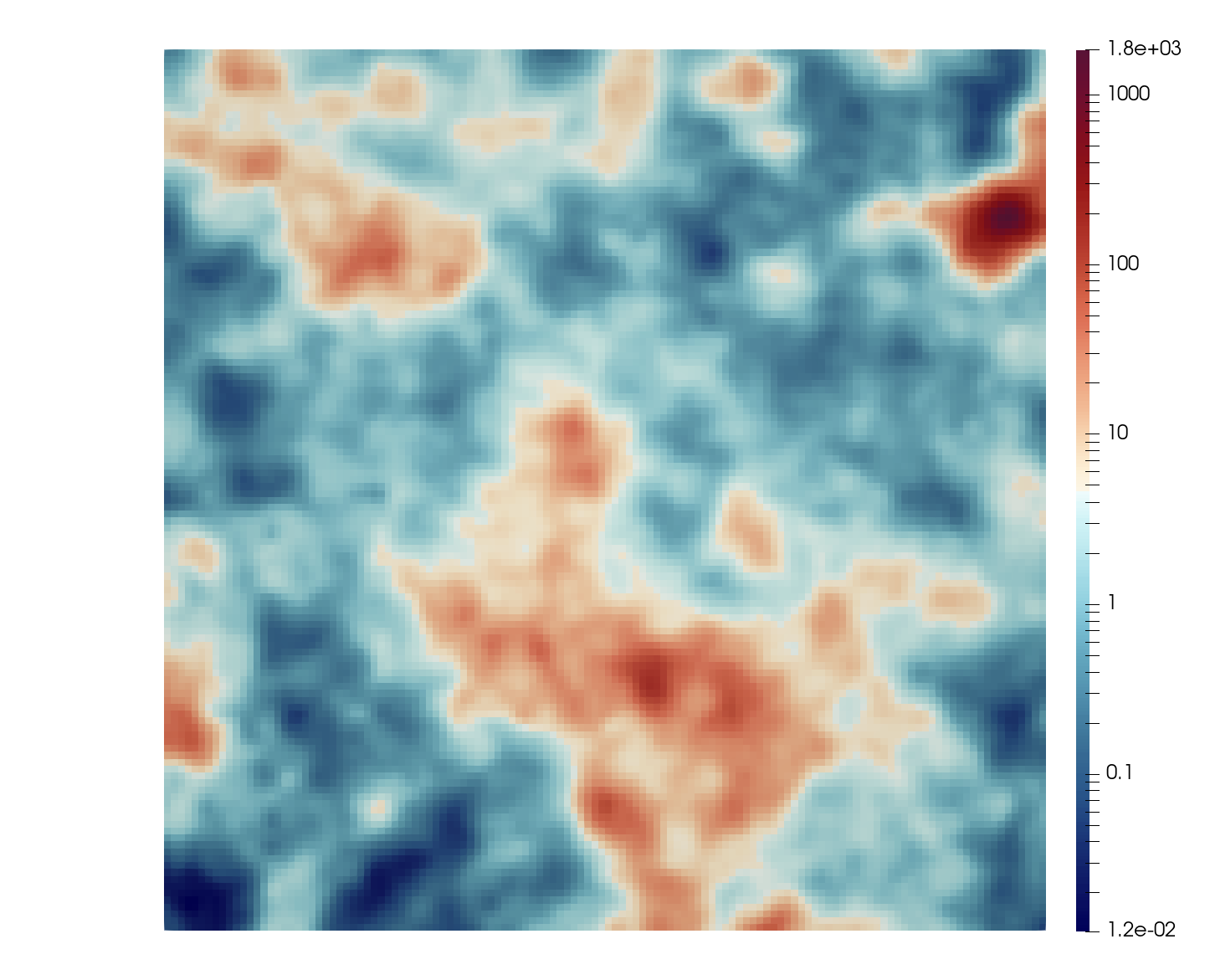}
        \includegraphics[trim={7cm 2cm 7cm 0.4cm}, clip, width=0.25\textwidth]{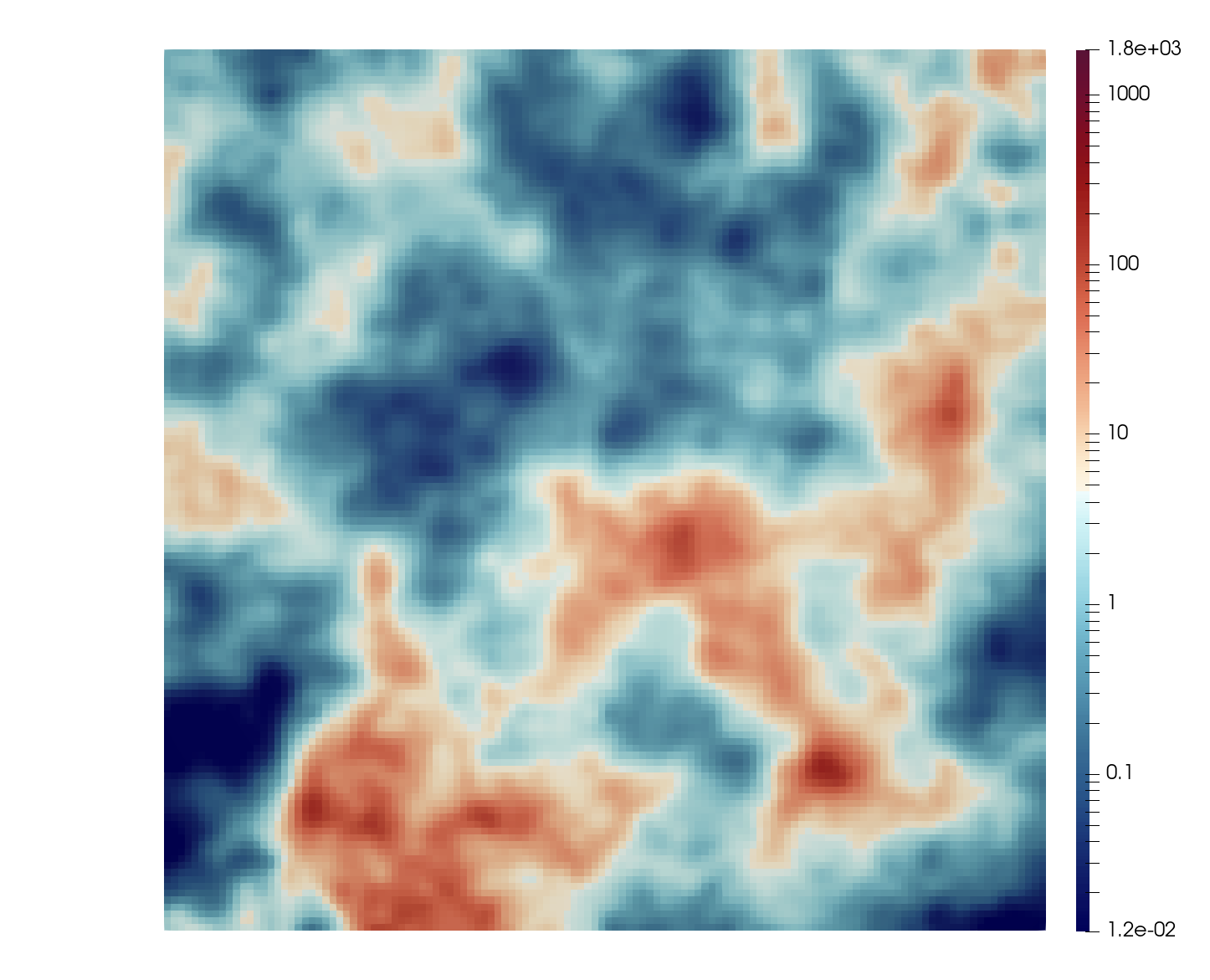}
        \hspace{0.1mm}
        \includegraphics[trim={7cm 2cm 7cm 0.4cm}, clip, width=0.25\textwidth]{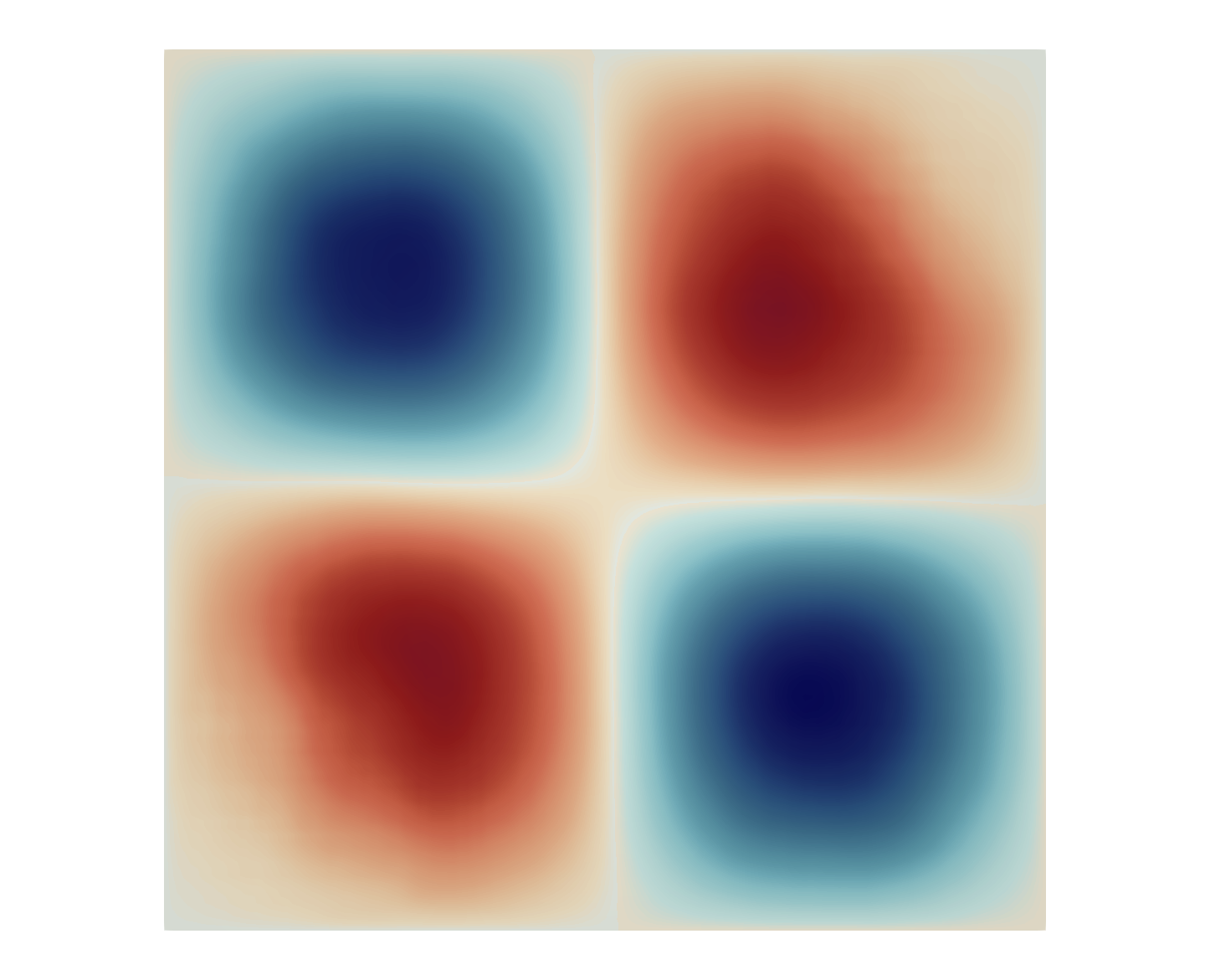}
        \raisebox{1.0in}{\rotatebox{-90}{$k=100$}}
    \end{center}
    \vspace*{-0.2cm}
    \caption{Left to right: Two GRF samples $\by_\ell^{(m)}$ on different mesh diameters;
    Computed control $\bz^{(k)}_\ell$ after $k=10$ and $k=100$ iterations of~\Cref{alg:mlsgd}.}
    \label{fig:log-normal-fields-and-control}
\end{figure}

    \section{Batched Stochastic Gradient Descent}
\label{sec:batched-stochastic-gradient-descent}
As a foundation and baseline for MLSGD, we first introduce the BSGD method.
We begin with an algorithmic description in~\Cref{subsec:bsgd-algorithm},
followed by the first numerical insights
and experiments in~\Cref{subsec:bsgd-experiments},
and conclude with a discussion of its convergence rate
and computational complexity in~\Cref{subsec:bsgd-discussion}.

    \subsection{Algorithm}
\label{subsec:bsgd-algorithm}

At its core, BSGD generates, similar to~\eqref{eq:sgd-iteration}, a minimizing sequence
$\tset{\bz_\ell^{(k)}}_{k=0}^K$ of approximations to find the optimum of~\Cref{problem:ocp}.
We refer to~\Cref{alg:bsgd} for a high-level functional pseudocode generating this sequence.

\begin{algorithm}
    \caption{Batched Stochastic Gradient Descent (BSGD)}
    \begin{align*}
        &\texttt{function BSGD}(\bz_{\ell}^{(0)} \!, \tset{t_k}_{k=0}^{K-1} \!, M) \colon \\[-1mm]
        &\quad
        \begin{cases}
            \texttt{for } k=0,\dots,K-1 \colon \\
            \quad
            \begin{cases}
                E^{\text{MC}}_M[\bq_{\ell}^{(k)}], \, J^{\text{MC}}_M(\bz_{\ell}^{(k)})
                &\!\leftarrow \texttt{BatchEstimation}(\bz_{\ell}^{(k)} \!, M) \\[1mm]
                \hspace{1.75cm} \bz_{\ell}^{(k+1)}
                &\!\leftarrow \pi_{Z}
                \big(\bz_{\ell}^{(k)} - t_k (\lambda \bz_{\ell}^{(k)} - E^{\text{MC}}_M [\bq_{\ell}^{(k)}]) \big)
            \end{cases} \\
            \texttt{return } \bz_{\ell}^{(K)}
        \end{cases} \\[1mm]
        &\texttt{function BatchEstimation}(\bz_{\ell}^{(k)} \!, M) \colon \\[-1mm]
        &\quad
        \begin{cases}
            \texttt{for } m=1,2,\dots,M \colon  \\
            \quad
            \begin{cases}
                \text{// Sampling method for } \omega \mapsto \by^{(m)}_{\ell}
                \text{ cf.~\cite{baumgarten2025budgeted, kutri2024dirichlet} for details} \\
                \by^{(m)}_\ell \hspace{2mm} \leftarrow [\hspace{25mm}]
                \begin{cases}
                    \qquad \qquad \vdots
                \end{cases} \\[3mm]
                \text{// Find state } \bu_{\ell}^{(m, k)} \text{ to control }
                \bz_{\ell}^{(k)} \text{ and realization } \by^{(m)}_\ell \\[1mm]
                \bu_{\ell}^{(m, k)} \leftarrow [\by^{(m)}_\ell, \hspace{10mm} \bz_{\ell}^{(k)}]
                \begin{cases}
                    \texttt{Find } \bu_{\ell}^{(m,k)} \in V_{\ell} \texttt{ such that:} \\
                    \quad \cG_{\ell} [\by^{(m)}_{\ell}] \,\, \bu_{\ell}^{(m, k)} = \bz_{\ell}^{(k)}
                \end{cases} \\[5mm]
                \text{// Find adjoint } \bq_{\ell}^{(m,k)} \text{ to state }
                \bu_{\ell}^{(m, k)} \text{ and realization } \by^{(m)}_\ell \\[1mm]
                \bq_{\ell}^{(m,k)} \leftarrow [\by^{(m)}_{\ell}, \bd - \bu_{\ell}^{(m, k)}]
                \begin{cases}
                    \texttt{Find } \bq_{\ell}^{(m,k)} \in V_{\ell} \texttt{ such that:} \\
                    \quad \cG^{*}_{\ell} [\by^{(m)}_{\ell}] \,\, \bq_{\ell}^{(m, k)} = \bd - \bu_{\ell}^{(m, k)}
                \end{cases}
            \end{cases} \\ \\[-2mm]
            \text{// Return result of estimators defined in \eqref{eq:bsgd-iteration} and \eqref{eq:ocp-objective-estimation}} \\
            \texttt{return } E^{\text{MC}}_M [\bq_{\ell}^{(k)}], \,\,  J^{\text{MC}}_M(\bz_{\ell}^{(k)})
        \end{cases}
    \end{align*}
    \label{alg:bsgd}
\end{algorithm}

\paragraph{BSGD function}
The algorithm starts in the \texttt{BSGD} function
taking an initial guess $\bz_{\ell}^{(0)}$,
an appropriate step size rule $\tset{t_k}_{k=0}^{K-1}$
(here directly given as a sequence of length $K$, which may be replaced by a function)
and the batch size $M$ as inputs.
The sequence $\tset{\bz_\ell^{(k)}}_{k=0}^K$ is generated by iteratively
solving~\eqref{eq:optimality-condition}
with estimates to $\EE_k[\bq^{(k)}] \in {\rL^2_k(\Omega, V)}$
\begin{equation}
    \label{eq:bsgd-iteration}
    \bz^{(k+1)}_\ell \!\leftarrow\! \pi_{Z} \big(\bz_{\ell}^{(k)}
    \!- t_k (\lambda \bz_{\ell}^{(k)} \!- E^{\text{MC}}_M[\bq_{\ell}^{(k)}]) \big) \\
    \,\,\, \text{with} \,\,\,
    E_{M}^{\text{MC}}[\bq_{\ell}^{(k)}] \coloneqq
    \frac{1}M \sum_{m=1}^M \bq_\ell^{(m, k)}.
\end{equation}
Here, we represent the batch estimation as a MC method using $M$
independent and identically distributed samples of
the approximated adjoint solutions $\bq_{\ell}^{(m, k)}$.
These samples are generated within the \texttt{BatchEstimation} function
called in each optimization step.

\paragraph{Batch estimation function}
Every call to \texttt{BatchEstimation} takes the current control
$\bz_{\ell}^{(k)}$ and the batch size $M$ as input.
Then, $M$ independent realizations of $\by_\ell^{(m)}$ are drawn
and used as input, together with the fixed $\bz_{\ell}^{(k)}$,
to approximate the solution of the systems~\eqref{eq:ocp-constraint} and~\eqref{eq:adjoint-pde}.
This computation can be fully parallelized for all samples,
provided sufficient memory and processing resources are available.
The function returns the Monte Carlo estimate of the adjoint solution,
$E^{\text{MC}}_M[\bq_{\ell}^{(k)}]$, to be used in~\eqref{eq:bsgd-iteration},
as well as an estimate of the objective~\eqref{eq:ocp-objective}
\begin{equation}
    \label{eq:ocp-objective-estimation}
    J^{\text{MC}}_M(\bz_{\ell}^{(k)}) \coloneqq \frac{1}{M} \sum_{m=1}^M \tfrac{1}{2} \bnorm{\bu_{\ell}^{(m, k)} - \bd}^2_{W}
    + \tfrac{\lambda}{2} \bnorm{\bz_{\ell}^{(k)}}_{W}^2
\end{equation}
using the state approximations $\bu_{\ell}^{(m, k)}$.
In machine learning, this is often referred to as
empirical risk~\cite{bottou2018optimization},
which we use here to experimentally monitor convergence
(see Subsections~\ref{subsec:bsgd-experiments},~\ref{subsec:mlsgd-experiments}, and~\ref{subsec:bmlsgd-experiments}).
However, in the convergence analysis, we prefer using $J(\bz_\ell^{(k)})$
to avoid introducing an additional estimation error.
Note that we also use $J(\bz^{(k)})$, for instance in~\Cref{lem:error-gradient-estimation},
to express the objective~\eqref{eq:ocp-objective} for a non-spatially discretized
control $\bz^{(k)} \in \rL^2_k(\Omega, Z)$ of~\eqref{eq:ocp-constraint},
which is not spatially discretized but still subject to stochastic approximation.

    \subsection{Experiments} \label{subsec:bsgd-experiments}

To convey the functionality of the BSGD method
and its implementation in M++~\cite{baumgarten2021parallel},
we present the first numerical experiments to evaluate
the method and to motivate potential improvements.

\paragraph{Investigation of the batch size}
We examine the influence of a variable batch size $M$ on a grid with mesh diameter $h_\ell = 2^{-7}$
combined with the constant step size $t_k \equiv 100$ and the initial control $\bz^{(0)}\equiv 0$.
\Cref{fig:batch-size-mc} illustrates two plots:
the left one presents estimates of $J^{\text{MC}}_M(\bz_{\ell}^{(k)})$
according to~\eqref{eq:ocp-objective-estimation} over the iteration $k$,
while the right plot shows evaluations of the gradient
used in~\eqref{eq:bsgd-iteration} in the $\rL^2$-norm
$\tnorm{E^{\text{MC}}_M [\bg_{\ell}^{(k)}]}_{\rL^2(\cD)}$
plotted against the total computing time.
A higher noise level in $J^{\text{MC}}_M(\bz_{\ell}^{(k)})$
is observed for smaller batch sizes $M$,
yet all methods oscillate around the same value.
The impact of the batch size on the gradient norm is even more pronounced:
the optimality condition~\eqref{eq:optimality-condition}
is visibly better satisfied with larger batches,
but this comes at the cost of an
increased computational effort.
Beyond this, a smoothed descent and arguably better convergence
rates are observed for larger batches.
Thus, if high accuracy is desired, larger batches are more reliable
in achieving the result.
However, as soon as we impose a tight limit on the computational cost,
the choice for the smaller batch might be preferable.

\begin{figure}
    \begin{center}
        \includegraphics[width=0.8\textwidth]{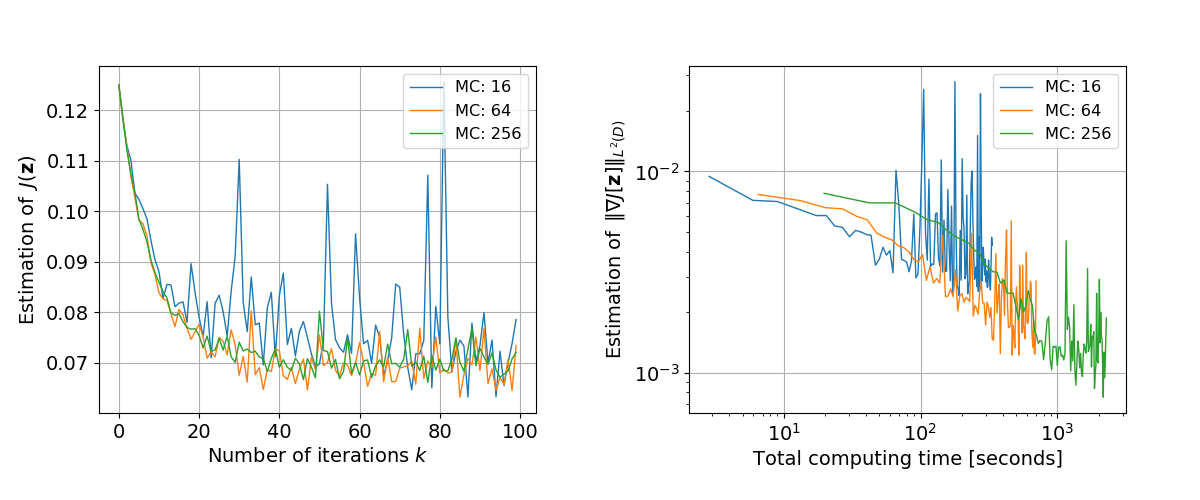}
    \end{center}

    \vspace{-0.4cm}

    \caption{Comparison of different batch sizes $M$.}
    \label{fig:batch-size-mc}
\end{figure}

\paragraph{Further experiments}
Running~\Cref{alg:bsgd} requires the choice of
several hyperparameters and adaptions to
the given computational environment.
This includes, besides the batch size $M$,
the appropriate choice of the discretization level $\ell$,
the step size $t_k$, the number of optimization steps $K$,
the parallelization strategy, and further configurations
of the involved numerical solvers and discretization schemes.
We omit a detailed presentation of our numerical
experiments conducted on this,
and refer to~\Cref{sec:adaptivity-and-budgeting}
where we develop an adaptive algorithm determining
several of the hyperparameters automatically.
Before we turn to this, we recall
the convergence and complexity theory of BSGD methods.

    \subsection{Properties and Discussion}
\label{subsec:bsgd-discussion}

As just experimentally observed in~\Cref{subsec:bsgd-experiments},
the batch size $M$ is crucial,
$M$ being small leads to a high variance in the gradient estimate,
while $M$ being large increases the computational cost.
This leads to a trade-off between the per-iteration cost
and expected per-iteration improvements.
To navigate this trade-off, we recall the convergence
properties of SAA and SA methods,
here distinguished through
the batch size $M^{\text{SAA}} \gg M^{\text{SA}}$.

\paragraph{Convergence and complexity}
Assuming strong $\mu$-convexity~\eqref{eq:strongly-convex}, {$\bz^* \in Z$}
satisfying~\eqref{eq:optimality-condition} a.e.~in $\cD$
and an appropriate choice for $t_k$ (e.g.~satisfying~\eqref{eq:step-size-assumption}),
an SGD method (including mini-batches of size $M^{\text{SA}}$) converges at a sub-linear rate
(cf.~\cite[Sec.~3.3.]{bottou2018optimization})
\begin{equation}
    \label{eq:sub-linear-convergence-bsgd}
    {\EE \big[}J(\bz^{(K)}_{\ell}) - J(\bz^*) {\big]} \leq \cO(K^{-1})
    \,\, \text{ at cost } \,\,
    \rC_{\epsilon}^{\text{SA}} \lesssim M^{\text{SA}} h_\ell^{-\gamma} \epsilon^{-1}.
\end{equation}
Here, $h_\ell$ is assumed to be sufficiently small
leading to a batch cost of $h_\ell^{-\gamma} M^{\text{SA}}$.
$\rC_{\epsilon}^{\text{SA}}$ represents the cost
required to reach an accuracy of at least $\epsilon > 0$
and $h_\ell^{-\gamma}$ denotes the cost per sample,
as assumed in~\eqref{eq:assumption-gamma-CT} and~\eqref{eq:assumption-gamma-mem}.

On the other hand, if the batch size is large enough to fall within the SAA regime,
classical optimization methods (e.g., conjugate gradient descent, BFGS, \dots)
achieve at least linear convergence.
Thus, for some $\rho \in (0, 1)$, we get
(cf.~\cite[Sec.~3.3.]{bottou2018optimization})
\begin{equation}
    \label{eq:linear-convergence-bsgd}
    \EE \big[J(\bz^{(K)}_{\ell}) - J(\bz^*) {\big]} \leq \cO(\rho^{K})
    \,\, \text{ at cost } \,\,
    \rC_{\epsilon}^{\text{SAA}} \lesssim M^{\text{SAA}} h_\ell^{-\gamma} \log (\epsilon^{-1})
\end{equation}
to reach an accuracy $0 < \epsilon < 1$.
Thus, SAA methods offer superior convergence rates compared to SA methods.
However, since $M^{\text{SAA}} \gg M^{\text{SA}}$,
each optimization step in SAA is significantly more expensive.
In cases where the input data exhibits high approximate similarity,
SA methods are in practice more efficient~\cite{bottou2018optimization}.
But, once the optimization error falls within the same order as the approximate similarity,
SAA methods become preferable, as SA methods can no longer guarantee an improvement in expectation.
We refer again to~\Cref{fig:batch-size-mc}
where aspects of this discussion can be observed, too.

\paragraph{Adaptivity and step sizes}
To combat the increased variance of SA methods,
but also to reduce the total computational cost of SAA methods,
adaptive batch sizes
(cf.~{\cite{chen2024minibatch, geiersbach2023optimization, van2019robust, beiser2023adaptive}} for optimal control),
adaptive step sizes~\cite{koehne2024adaptivestepsizespreconditioned}
and the ADAM optimizer~\cite{kingma2014adam} have been proposed,
giving more robust paths of iterates $\bz_\ell^{(k)}$,
objectives $J^{\text{MC}}_M(\bz_{\ell}^{(k)})$
and gradient evaluations.
This also enables a stopping criterion
other than a maximum number of optimization steps $K$.
We also note that convergence can still be shown for the convex case,
by considering the averaging scheme of~\cite{polyak1992averaging},
for details see~\cite{geiersbach2020stochastic}.
This also requires a decreasing or sufficiently small sequence of $t_k$,
which has to be chosen carefully to ensure convergence.

    \section{Multilevel Monte Carlo}\label{sec:multilevel-monte-carlo}
As a remedy to the observed challenges in the previous section,
we propose the use of an MLMC estimator to determine
the expected discretized adjoint solution in each optimization step $k$.
Expanding $\EE_k[\bq_L^{(k)}]$ in a telescoping sum over discretization levels $\ell=0, \dots, L$ of the FE mesh,
inserting equation~\eqref{eq:definition-vell}
and using the isometric transfer operators
$\rP_{\ell}^L \colon V_\ell \rightarrow V_L$,
constructed through $\rP_{\ell}^L = \rP_{L-1}^L \dots \rP_{1}^2 \rP_{0}^1$, gives
\begin{align*}
    \EE_k[\bq_L^{(k)}] = \rP_{0}^L \EE_k[\bq^{(k)}_0]
    + \sum_{\ell=1}^{L} \rP_{\ell}^L \EE_k[\bq^{(k)}_{\ell} - \rP_{\ell-1}^\ell \bq^{(k)}_{\ell-1}]
    = \sum_{\ell=0}^L \EE_k[\bp^{(k)}_{\ell}]
    \, \approx \, \EE_k[\bq^{(k)}].
\end{align*}
This motivates the MLMC estimator (cf.~\cite{baumgarten2025budgeted} for details in similar notation)
\begin{equation}
    \label{eq:mlmc-estimator}
    E^{\text{ML}}[\bq_L^{(k)}] \coloneqq \sum_{\ell=0}^L \rP_{\ell}^{L} E^{\text{MC}}_{M_\ell}[\bp_\ell^{(k)}]
    \quad \text{with} \quad
    E^{\text{MC}}_{M_\ell}[\bp_\ell^{(k)}] \coloneqq \frac{1}{M_\ell} \sum_{m=1}^{M_\ell} \bp_\ell^{(m, k)},
\end{equation}
where $\set{M_\ell}_{\ell=0}^L$ denotes the number of samples on each level.
Note that for $\ell \neq 0$ the difference in~\eqref{eq:definition-vell}
is the same sample computed on two different levels, i.e.,
$\bp_\ell^{(m,k)} = \bq_\ell^{(m,k)} - \rP_{\ell-1}^{\ell} \bq_{\ell-1}^{(m,k)}$
(cf.~\Cref{fig:log-normal-fields-and-control} for an illustration of the same input data on two different levels).
The mean squared error in the above approximation is expressed by
(cf.~e.g.~\cite[Theorem 3.1]{bierig2015convergence} for a similar setting)
\begin{equation}
    \label{eq:mse-mlmc}
    \mathrm{err}^{\mathrm{MSE}}_k \big(E^{\text{ML}}[\bq_L^{(k)}]\big)
    = \underbrace{\sum_{\ell=0}^L M_{\ell}^{-1}
        \bnorm{{\bp_{\ell}^{(k)} - \EE_k[\bp_\ell^{(k)}]}}^2_{\rL^2_k(\Omega, W)}}_{\eqqcolon {\err}_k^{\mathrm{sam}}}
    + \,{\underbrace{\bnorm{\EE_k[\bq_{L}^{(k)} - \bq^{(k)}]}_W^2}_{\eqqcolon {\err}_k^{\mathrm{num}}}}.
\end{equation}
If~\eqref{eq:assumption-alpha-q} and~\eqref{eq:assumption-beta-q} are satisfied, we can control both
the sampling error ${\err}^{\mathrm{sam}}_k$ and the numerical error ${\err}^{\mathrm{num}}_k$,
either by increasing the level-dependent number of samples $M_\ell$ or by refining the mesh.
Naturally, both measures come at an increased computational cost captured
by~\eqref{eq:assumption-gamma-CT} and~\eqref{eq:assumption-gamma-mem}.
However, the multilevel structure allows for the construction of a sample sequence $\set{M_\ell}_{\ell=0}^L$,
such that the computational cost $\rC_\epsilon$ (measured in memory or computing-time)
to reach a desired accuracy of $0 < \epsilon < \re^{-1}$ is bounded by
(cf.~\cite[Theorem 1]{cliffe2011multilevel} and~\cite[Theorem 3.1]{giles2008multilevel})
\begin{equation}
    \label{eq:mlmc-complexity}
    \sqrt{\err^{\mathrm{MSE}} \big(E^{\text{ML}}[\bq_L^{(k)}]}\big) \overset{!}{<} \epsilon
    \quad \Rightarrow \quad
    \rC_\epsilon \big(E^{\text{ML}}[\bq_L^{(k)}]\big) \lesssim
    \begin{cases}
        \epsilon^{-2} & \beta > \gamma, \\
        \epsilon^{-2} (\log(\epsilon))^2 & \beta = \gamma, \\
        \epsilon^{-2-(\gamma-\beta)/\alpha} & \beta < \gamma.
    \end{cases}
\end{equation}
This stands in contrast to the combination of FE with standard MC methods which has the cost of
$\rC_\epsilon \big({E^{\text{MC}}_M [\bq_L^{(k)}]}\big) \lesssim \epsilon^{-2- \gamma/\alpha}$,
giving at least an improvement of $\beta/\alpha$ in the rate
for the same accuracy~\cite{cliffe2011multilevel}.
The central idea of this paper is to leverage this improvement for the batch estimation within an SGD method.
Before we outline this approach in further detail in~\Cref{sec:multilevel-stochastic-gradient-descent},
we briefly state how to estimate $\err^{\mathrm{sam}}$ and $\err^{\mathrm{num}}$.
This is done by a posteriori error estimators motivated in~\cite{giles2015multilevel}
via regression and extrapolation arguments.
In~\cite{baumgarten2025budgeted} the adaption to full field estimates is presented
\begin{equation}
    \label{eq:squared-bias-estimate-mlmc}
    \widehat{\err}^{\text{num}}_k \! \coloneqq \!\!
    \max_{\ell = 1, \dots, L} \!\!
    \roundlr{\!\frac{{\|E_{M_{\ell}}^{\text{MC}}[\bp_{\ell}^{(k)}]\|_W}}{(2^{\widehat{\alpha}} - 1) 2^{{\widehat{\alpha} (L - \ell)}}}\!}^2
    \!\text{with} \,\,
    \argmin_{(\widehat{\alpha}, \widehat{c})} \sum_{\ell=1}^L
    \! \big( \! \log_2 {\|E_{M_{\ell}}^{\text{MC}}[\bp_{\ell}^{(k)}]\|_W} \! + \widehat{\alpha} \ell - \widehat{c} \big)^2.
\end{equation}
Thus $\widehat{\alpha}$ is determined by regression
and $\widehat{\err}^{\text{num}}_k$ follows from extrapolation arguments.
The estimation of the sampling error is similar to computing a sample variance, i.e.,
\begin{equation}
    \label{eq:sampling-error-estimator}
    \widehat{\err}^{\mathrm{sam}}_k \coloneqq \sum_{\ell=0}^L \frac{s^2_{\ell}[\bp_\ell^{(k)}]}{M_{\ell} (M_{\ell} - 1)}
    \,\,\,\,
    \text{with}
    \,\,\,\,
    s^2_{\ell}[\bp_\ell^{(k)}] \coloneqq \sum_{m=1}^{M_{\ell}}
    \bnorm{\bp_{\ell}^{(m,k)} - E^{\text{MC}}_{M_{\ell}}[\bp_{\ell}^{(k)}]}_W^2.
\end{equation}

    \section{Multilevel Stochastic Gradient Descent}
\label{sec:multilevel-stochastic-gradient-descent}

Based on the presented MLMC and BSGD method,
we now introduce one main contribution of this work,
the Multilevel Stochastic Gradient Descent (MLSGD) method.
We outline the algorithm in~\Cref{subsec:mlsgd-algorithm},
present the first experimental results in~\Cref{subsec:mlsgd-experiments}
and analyze its convergence and complexity in~\Cref{subsec:mlsgd-analysis}.

    \subsection{Algorithm}
\label{subsec:mlsgd-algorithm}

To explain~\Cref{alg:mlsgd}, we first note its similarity
to the previously discussed BSGD method outlined in~\Cref{alg:bsgd}
and draw on knowledge from~\Cref{sec:multilevel-monte-carlo}
to explain the differences and advantages.
Again, we follow a function-wise explanation of the algorithm.

\paragraph{The MLSGD function}
The \texttt{MLSGD} function serves as the entry point to~\Cref{alg:mlsgd} and
accepts an initial guess for $\bz_{L}^{(0)}$, an appropriate sequence of step sizes $\tset{t_k}_{k=0}^{K-1}$,
and a multilevel batch $\tset{M_\ell}_{\ell=0}^L$ as input arguments.
Similar to~\Cref{alg:bsgd}, the MLSGD method iteratively optimizes the control $\bz_{L}^{(k)}$.
However, instead of using~\eqref{eq:bsgd-iteration}, we employ the iteration scheme
\begin{equation}
    \label{eq:mlsgd-iteration}
    \bz_{L}^{(k+1)} \leftarrow \pi_{Z} \big(\bz_{L}^{(k)} - t_k E^{\text{ML}}[\bg^{(k)}_{L}] \big)
    \quad \text{with} \quad
    E^{\text{ML}}[\bg_{L}^{(k)}] \coloneqq \lambda \bz_{L}^{(k)} - E^{\text{ML}}[\bq_{L}^{(k)}]
\end{equation}
to approximate the solution to~\Cref{problem:ocp}.
We note that the gradient estimation $E^{\text{ML}}[\bg_{L}^{(k)}]$
is now computed using $E^{\text{ML}}[\bq_{L}^{(k)}]$,
as given through~\eqref{eq:mlmc-estimator}.

\paragraph{Multilevel batch estimation}
To compute the adjoint $E^{\text{ML}}[\bq_{L}^{(k)}]$,
the function \texttt{MultiLevelEstimation} is called in each optimization step,
taking the current control $\bz_{L}^{(k)}$ and a multilevel batch $\tset{M_\ell}_{\ell=0}^L$ as input arguments.
By~\eqref{eq:definition-vell}, the adjoint and the state are computed
with the \texttt{BatchEstimation} function, as in~\Cref{alg:bsgd},
on the lowest level $\ell=0$.
For the higher levels $\ell=1,\dots,L$, the estimation is performed
for the level pair $(\ell, \ell-1)$ using the function \texttt{LevelPairEstimation}.
This also motivates the definitions of $\by_{\ell, \ell-1}^{(m, k)}$, $\bu_{\ell, \ell-1}^{(m, k)}$
and $\bq_{\ell, \ell-1}^{(m, k)}$ in the top row of~\Cref{alg:mlsgd}.

\paragraph{Level pair batch estimation}
Similar to the \texttt{BatchEstimation} function in~\Cref{alg:bsgd},
the \texttt{LevelPairEstimation} function approximates the
state~\eqref{eq:ocp-constraint} and adjoint~\eqref{eq:adjoint-pde} systems,
taking the current control $\bz_{L}^{(k)}$, the level-specific batch size $M_{\ell}$,
and the level $\ell$ as input arguments.
Since the control $\bz_{L}^{(k)}$ is always stored at the highest level $L$,
the first step is to restrict it to levels $\ell$ and $\ell-1$
with the restriction operators $\rR^{\ell}_L \colon V_L \rightarrow V_\ell$
and $\rR^{\ell-1}_L \colon V_L \rightarrow V_{\ell-1}$, respectively.
Additionally, the control must be distributed across the
multiple processes executing the batch loop in parallel.
We omit the details on the parallelization and refer
to~\cite{baumgarten2024fully, baumgarten2025budgeted}
as well as to a short discussion in~\Cref{subsec:mlsgd-experiments}.
Having the distributed and restricted controls $\bz_{\ell, \ell-1}^{(k)}$,
we can approximate the state and adjoint on both levels $\ell$ and $\ell-1$,
taking the realizations $\by_{\ell, \ell-1}^{(m, k)}$ as input.
Finally, the function returns with
$\bp_{\ell}^{(m,k)} = \bq_{\ell}^{(m,k)} - \rP_{\ell-1}^\ell \bq_{\ell-1}^{(m,k)}$
the MC estimate of the adjoint level difference as in~\eqref{eq:mlmc-estimator},
and with $\rQ_{\ell}^{(m, k)} \coloneqq \tfrac{1}{2} \bnorm{\bu_{\ell}^{(m, k)} - \bd}^2_{W}$,
$\rY_{\ell}^{(m, k)} \coloneqq \rQ_{\ell}^{(m, k)} - \rQ_{\ell-1}^{(m, k)}$ for $\ell \geq 1$
and $\rY_{0}^{(m, k)} \coloneqq \rQ_{0}^{(m, k)}$ for $\ell = 0$ the estimate
\begin{equation}
    \label{eq:ocp-objective-level-pair-estimation}
    J^{\text{MC}}_{M_\ell}(\bz_{\ell, \ell-1}^{(k)})
    \coloneqq \frac{1}{M_\ell} \sum_{m=1}^{M_\ell} \rY_{\ell}^{(m, k)}.
\end{equation}
This enables the estimation for~\eqref{eq:ocp-objective},
then returned by the \texttt{MultiLevelEstimation} function
\begin{equation}
    \label{eq:ocp-objective-multilevel-estimation}
    J^{\text{ML}}(\bz_{L}^{(k)})
    \coloneqq \sum_{\ell=0}^L J^{\text{MC}}_{M_\ell}(\bz_{\ell, \ell-1}^{(k)})
    + \tfrac{\lambda}{2} \bnorm{\bz_{L}^{(k)}}_{W}^2.
\end{equation}

In conclusion,~\Cref{alg:mlsgd} utilizes a MLMC method
as replacement for the batch estimator.
Care has to be taken to ensure that the control is distributed and restricted
to the correct data structures; and that the level pairs are solved
with the same input realizations.
The overall algorithm, however, is still strongly related to the BSGD method.

\begin{algorithm}
    \caption{Multilevel Stochastic Gradient Descent (MLSGD)}
    \label{alg:mlsgd}
    \begin{align*}
        &\texttt{def }
        \by_{\ell, \ell - 1}^{(m)} \coloneqq (\by^{(m)}_{\ell} \!\!, \by_{\ell - 1}^{(m)}), \,\,
        \bu_{\ell, \ell - 1}^{(m,k)} \coloneqq (\bu_{\ell}^{(m,k)} \!\!, \bu_{\ell - 1}^{(m,k)}), \,\,
        \bq_{\ell, \ell - 1}^{(m,k)} \coloneqq (\bq_{\ell}^{(m,k)} \!\!, \bq_{\ell - 1}^{(m,k)}) \\[4mm]
        &\texttt{function MLSGD}(\bz_{L}^{(0)} \!, \tset{M_\ell}_{\ell=0}^{L}, \tset{t_k}_{k=0}^{K-1}) \colon \\[-1mm]
        &\quad
        \begin{cases}
            \texttt{for } k=0,\dots,K-1 \colon \\
            \quad
            \begin{cases}
                E^{\text{ML}}[\bq_{L}^{(k)}], \, J^{\text{ML}}(\bz_{L}^{(k)})
                &\!\leftarrow \texttt{MultiLevelEstimation}(\bz_{L}^{(k)} \!, \tset{M_{\ell}}_{\ell=0}^{L}) \\[1mm]
                \hspace{0.65cm} \bz_{L}^{(k+1)} \!, \, E^{\text{ML}}[\bg_{L}^{(k)}]
                &\!\leftarrow {\texttt{GradientDescent}}(\bz_{L}^{(k)} \!, E^{\text{ML}}[\bq_L^{(k)}], t_k)
            \end{cases} \\
            \texttt{return } \bz_{L}^{(K)}
        \end{cases} \\ \\[-2mm]
        &\texttt{function GradientDescent}(\bz_{L}^{(k)} \!, E^{\text{ML}}[\bq_L^{(k)}], t_k) \colon \\[-1mm]
        &\quad
        \begin{cases}
            E^{\text{ML}} [\bg_{L}^{(k)}]
            &\!\leftarrow \lambda \bz_{L}^{(k)} - E^{\text{ML}} [\bq_{L}^{(k)}] \\[1mm]
            \bz_{L}^{(k+1)}
            &\!\leftarrow \pi_{Z}
            \big(\bz_{L}^{(k)} - t_k E^{\text{ML}} [\bg_{L}^{(k)}] \big) \\[1mm]
            \texttt{return } &\hspace{0.4cm} \bz^{(k+1)}, \,\, E^{\text{ML}} [\bg_{L}^{(k)}]
        \end{cases} \\ \\[-2mm]
        &\texttt{function MultiLevelEstimation}(\bz_{L}^{(k)} \!, \tset{M_{\ell}}_{\ell=0}^{L}) \colon \\[-1mm]
        &\quad
        \begin{cases}
            \text{// Solve state and adjoint system on lowest level } \ell=0 \text{ (cf.~\Cref{alg:bsgd})} \\[1mm]
            E^{\text{MC}}_{M_0}[\bp_{0}^{(k)}], \,\, J^{\text{MC}}_{M_0}(\bz_{0}^{(k)})
            \leftarrow \texttt{BatchEstimation}(\bz_{L}^{(k)} \!, M_{0}) \\
            \texttt{for } \ell = 1, \dots, L \colon \\
            \quad
            \begin{cases}
                \text{// Solve state and adjoint system for level pair } (\ell, \ell - 1) \\[1mm]
                E^{\text{MC}}_{M_\ell}[\bp_{\ell}^{(k)}], \,\, J^{\text{MC}}_{M_\ell}(\bz_{\ell,\ell-1}^{(k)})
                \leftarrow \texttt{LevelPairEstimation}(\bz_{L}^{(k)} \!, M_{\ell}, \, \ell) \\[1mm]
            \end{cases} \\ \\[-3mm]
            \text{// Return result of multilevel sums as in \eqref{eq:mlmc-estimator} and \eqref{eq:ocp-objective-multilevel-estimation}} \\
            \texttt{return }
            \sum_{\ell=0}^L \rP_{\ell}^{L} E^{\text{MC}}_{M_\ell}[\bp_\ell^{(k)}], \,\,\,\,
            \sum_{\ell=0}^L J^{\text{MC}}_{M_\ell}(\bz_{\ell, \ell-1}^{(k)})
            + \tfrac{\lambda}{2} \bnorm{\bz_{L}^{(k)}}_{W}^2
        \end{cases} \\ \\[-2mm]
        &\texttt{function } \texttt{LevelPairEstimation}(\bz_{L}^{(k)} \!, M_{\ell}, \, \ell) \colon \\[-1mm]
        &\quad
        \begin{cases}
            \bz^{(k)}_{\ell, \ell-1} \leftarrow (\rR^{\ell}_L \bz_{L}^{(k)}, \hspace{1mm} \rR^{\ell-1}_L \bz_{L}^{(k)})
            \hspace{12mm} \text{// Restrict and distribute control} \\[1mm]
            \texttt{for } m=1,2,\dots,M_{\ell} \colon
            \hspace{19.5mm} \text{// Run in parallel with optimal distr.} \\[1mm]
            \quad
            \begin{cases}
                \text{// Sampling method for } \omega \mapsto \by_{\ell, \ell-1}^{(m)} \\
                \by^{(m)}_{\ell, \ell-1} \hspace{0.1mm} \leftarrow [\hspace{27mm}]
                \begin{cases}
                    \qquad \qquad \vdots
                \end{cases} \\[3mm]
                \text{// Find states } \bu_{\ell, \ell-1}^{(m, k)} \text{ to controls }
                \bz^{(k)}_{\ell, \ell-1} \text{ and realizations } \by^{(m)}_{\ell, \ell-1} \\
                \bu_{\ell, \ell-1}^{(m, k)} \leftarrow
                [\by^{(m)}_{\ell, \ell-1}, \hspace{7mm} \bz^{(k)}_{\ell, \ell-1}]
                \begin{cases}
                    \texttt{Find } \bu_{\ell, \ell-1}^{(m,k)} \in V_{\ell, \ell-1} \texttt{ such that:} \\
                    \quad \cG_{\ell, \ell-1} [\by^{(m)}_{\ell, \ell-1}] \,\,
                    \bu_{\ell, \ell-1}^{(m, k)} = \bz_{\ell, \ell-1}^{(k)}
                \end{cases} \\ \\[-3mm]
                \text{// Find adjoints } \bq_{\ell, \ell - 1}^{(m,k)} \text{ to states }
                \bu_{\ell, \ell-1}^{(m, k)} \text{ and realizations } \by^{(m)}_{\ell, \ell-1} \\
                \bq_{\ell, \ell - 1}^{(m,k)} \leftarrow [\by^{(m)}_{\ell, \ell-1}, \bd - \bu_{\ell, \ell-1}^{(m, k)}]
                \begin{cases}
                    \texttt{Find } \bq_{\ell, \ell-1}^{(m,k)} \in V_{\ell, \ell-1} \texttt{ such that:} \\
                    \quad \cG^{*}_{\ell, \ell-1} [\by^{(m)}_{\ell, \ell-1}] \,\, \bq_{\ell, \ell-1}^{(m, k)} = \bd - \bu_{\ell, \ell-1}^{(m, k)}
                \end{cases}
            \end{cases} \\ \\[-3mm]
            \text{// Return result of estimators defined in \eqref{eq:mlmc-estimator} and \eqref{eq:ocp-objective-level-pair-estimation}} \\
            \texttt{return } E^{\text{MC}}_{M_\ell}[\bp_\ell^{(k)}], \,\,\,\,
            J^{\text{MC}}_{M_\ell}(\bz_{\ell, \ell-1}^{(k)})
        \end{cases}
    \end{align*}
\end{algorithm}

    \subsection{Experiments}
\label{subsec:mlsgd-experiments}
Having introduced~\Cref{alg:mlsgd}, we present first experiments to evaluate the method.
Particularly, we compare the MLSGD with the BSGD method
and further investigate the influence of the multilevel batch $\set{M_\ell}_{\ell=0}^L$
at hand of the PDE example introduced in~\Cref{subsec:example-problem}.

\paragraph{Comparing multilevel estimation and batch estimation}
The layout of~\Cref{fig:multilevel-vs-batch}
follows the figures in~\Cref{subsec:bsgd-experiments},
illustrating a comparison between MLMC and MC gradient estimation.
Both methods use the constant step size $t_k \equiv 100$,
the same initial control $\bz^{(0)}\equiv 0$,
the same number of iterations $K=100$,
and the same number of CPUs $\abs{\cP}=64$.
The mesh resolution starts with $h_0=2^{-4}$ and ends with $h_L=2^{-7}$ for the MLMC case,
whereas MC estimation operates as in~\Cref{subsec:bsgd-experiments} on $h_\ell=2^{-7}$.
Neither of the estimated objectives, $J^{\text{ML}}(\bz_{L}^{(k)})$
and $J^{\text{MC}}_M(\bz_{L}^{(k)})$ in~\Cref{fig:multilevel-vs-batch},
shows a clear advantage over the iterations $k$,
as both oscillate around the same value with a similar noise level.
However, the plot of the gradient norms reveals a significant
speedup—greater than factor 10 in this example—for the MLMC gradient estimation.
The MLSGD method achieves a similar convergence rate
and gradient quality by offloading variance reduction
through batching on the lower discretization levels
at significantly reduced computational cost.
This serves as an initial indication that MLMC
is a promising alternative to batch estimation.

\begin{figure}
    \begin{center}
        \includegraphics[width=0.8\textwidth]{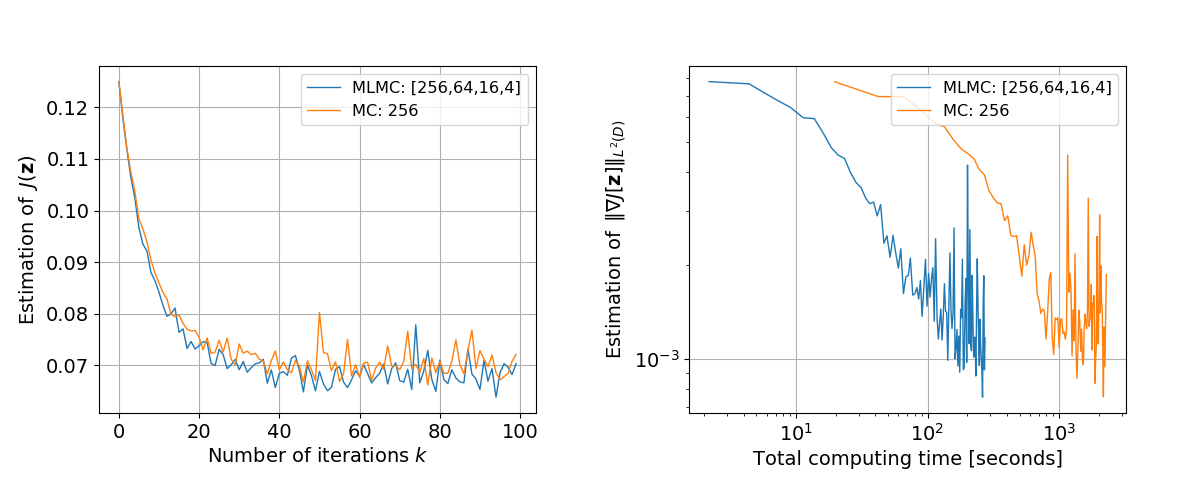}
    \end{center}

    \vspace{-0.6cm}

    \caption{Comparison of MC and MLMC gradient estimation.}
    \label{fig:multilevel-vs-batch}
\end{figure}

\paragraph{Investigation of the multilevel batch size}
Next, we examine the influence of the multilevel batch size
$\set{M_\ell}_{\ell=0}^L$ illustrated in~\Cref{fig:multilevel-batch-size}.
The blue lines in~\Cref{fig:multilevel-vs-batch} and~\Cref{fig:multilevel-batch-size}
correspond to the same batch with $h_0=2^{-4}$ and $h_L=2^{-7}$,
whereas the orange and green lines in~\Cref{fig:multilevel-batch-size}
also include $h_L=2^{-8}$ and $h_L=2^{-9}$, respectively.
We note that larger multilevel batches yield smoother and more robust
estimates of the objective $J^{\text{ML}}(\bz_{L}^{(k)})$.
The gradient norm plotted over computing time motivates
a similar discussion as in~\Cref{subsec:bsgd-experiments},
i.e., the larger the multilevel batch,
the better is~\eqref{eq:optimality-condition} satisfied
due to the reduced variance, but this time also due to the reduced bias.
Yet, the higher quality of the computational
results naturally comes at increased cost.

\begin{figure}
    \begin{center}
        \includegraphics[width=0.8\textwidth]{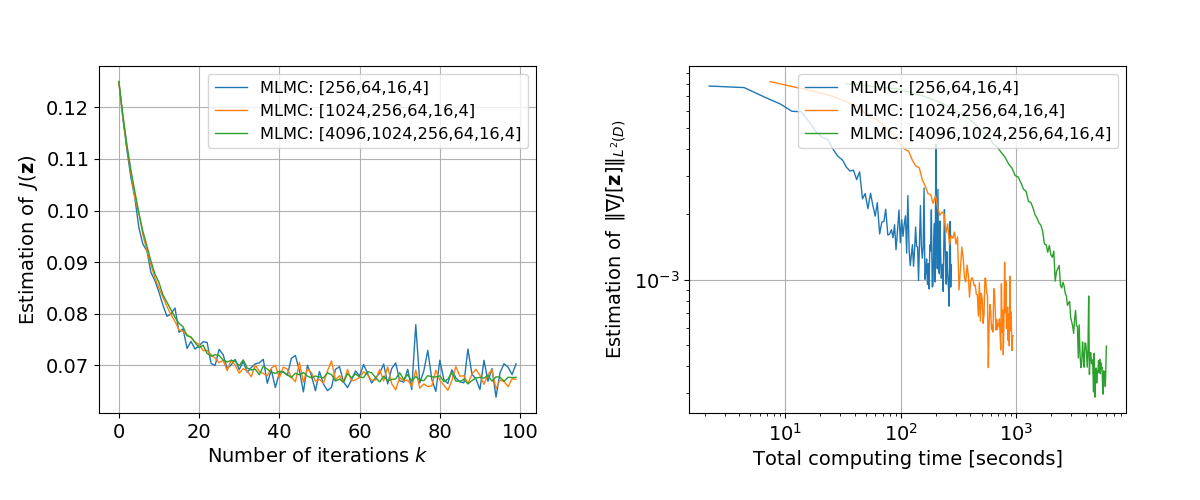}
    \end{center}

    \vspace{-0.6cm}

    \caption{Comparison of different multilevel batch sizes $\tset{M_\ell}_{\ell=0}^L$.}
    \label{fig:multilevel-batch-size}
\end{figure}

\paragraph{Discussion and remaining challenges}
Though the MLSGD method shows promising results in the numerical experiments,
several questions remain as well as new challenges arise:
(i)
We note that picking the optimal multilevel batch size
still depends upon the available computational resources.
Smaller batches are still to be favored for small computational time budgets.
In~\Cref{sec:adaptivity-and-budgeting}, we develop a hardware-aware method
that automatically finds the optimal batch size for a
given computational budget in CPU-time and memory.
(ii)
To incorporate the computational budgets,
an adaptive strategy for the step size $t_k$,
the multilevel batch $\tset{M_\ell}_{\ell=0}^L$,
the total number of iterations $K$
and the largest level $L$ has to be developed.
Before we propose ways to approach (i) and (ii)
algorithmically in~\Cref{sec:adaptivity-and-budgeting},
we examine the method's convergence and complexity
behavior analytically in~\Cref{subsec:mlsgd-analysis}.
(iii)
Lastly, leveraging parallel computing resources within a
multilevel setting is a non-trivial task.
The reason is that the MC estimates on the lower levels
are much more efficient with a sample parallelization,
whereas the MC estimates on the higher levels
often require a parallelization over the spatial domain $\cD$.
Yet, all estimates must be synchronized to compute the adjoint estimate
in~\eqref{eq:mlmc-estimator}.
To leverage the full potential of MLMC estimation
without sacrificing parallel efficiency,
careful algorithmic design is required.
We use the proposed multiindex data structure
of~\cite{baumgarten2025budgeted},
in which the discretization level~$\ell$ is
paired with a communication index~$s$.
The index~$s$ is chosen to minimize inter-processor communication,
enabling optimal parallelization when computing full field estimates.

    \subsection{Analysis}
\label{subsec:mlsgd-analysis}
Finally, within this section, we present comprehensive convergence and complexity
analysis of the MLSGD method.
The main result is given in~\Cref{thm:convergence-mlsgd},
leveraging~\Cref{lem:error-gradient-estimation} as a key idea.
\Cref{cor:convergence-mlsgd} generalizes the convergence result
on the objective and on the gradient.

\begin{lemma}
    \label{lem:error-gradient-estimation}
    There exists a multilevel batch $\tset{M_\ell}_{\ell=0}^L$,
    such that the error $\br_L^{(k)}$ between the gradient
    $\nabla J(\bz^{(k)})$ in step $k$ and its estimation
    $E^{\text{ML}}[\bg_L^{(k)}]$ in~\eqref{eq:mlsgd-iteration}
    \begin{equation}
        \label{eq:error-gradient-estimation}
        \br_L^{(k)} \coloneqq E^{\text{ML}}[\bg_L^{(k)}] - \nabla J(\bz^{(k)})
    \end{equation}
    as well as $\bnorm{\bz_L^{(k)} - \bz^{(k)}}_{\rL^2_k(\Omega, W)}$
    can be bounded by $\epsilon_k > 0$, particularly
    \begin{equation}
        \label{eq:gradient-control}
        \bnorm{\br_L^{(k)}}_{\rL^2_k(\Omega, W)}^2 \leq 2 \, \epsilon_k^2
        \quad \text{and} \quad
        \bnorm{\bz_L^{(k)} - \bz^{(k)}}_{\rL_k^2(\Omega, W)}^2 \leq \frac{c_{\cG} \, (1 - \theta)}{\lambda^2 c_{\cG} + 1} \, \epsilon_k^2,
    \end{equation}
    with $\theta \in (0, 1)$ and $c_{\cG}$ from~\eqref{eq:assumption-alpha-z} at computational complexity
    \begin{equation}
        \label{eq:gradient-complexity}
        \rC_k \roundlr{E^{\text{ML}}[\bg_L]} \lesssim
        \begin{cases}
            \epsilon_k^{-2} & \beta > \gamma, \\
            \epsilon_k^{-2} (\log(\epsilon_k))^2 & \beta = \gamma, \\
            \epsilon_k^{-2-(\gamma-\beta)/\alpha} & \beta < \gamma,
        \end{cases}
    \end{equation}
    where $\beta$ and $\gamma$ describe the decay of the sampling error
    and the increase of the computational cost as
    in~\eqref{eq:assumption-beta-q} and~\eqref{eq:assumption-gamma-CT}.
\end{lemma}

\begin{proof}
    By $\nabla J(\bz^{(k)}) = \lambda \bz^{(k)} - \EE_{k}[\bq^{(k)}]$,~\eqref{eq:mlsgd-iteration} and the triangle inequality, we have
    \begin{align*}
        \bnorm{\br_L^{(k)}}_{\rL_k^2(\Omega, W)}^2
        &= \EE_k \left[\bnorm{\lambda \bz_L^{(k)}  - \lambda \bz^{(k)} + \EE_k[\bq^{(k)}] - E^{\text{ML}}[\bq_L^{(k)}]}_W^2\right] \\
        &\leq 2 \lambda^2 \bnorm{\bz_L^{(k)} - \bz^{(k)}}_{\rL_k^2(\Omega, W)}^2
        + 2 \, \bnorm{\EE_k[\bq^{(k)}] - E^{\text{ML}}[\bq_L^{(k)}]}_{\rL^2_k(\Omega, W)}^2
    \end{align*}
    Since $\mathrm{err}^{\mathrm{MSE}}_k \coloneqq \EE_k \squarelr{\bnorm{\EE_k[\bq^{(k)}] - E^{\text{ML}}[\bq_L^{(k)}]}_W^2}$,
    we have with the decomposition~\eqref{eq:mse-mlmc}
    \begin{align*}
        \tfrac{1}{2} \bnorm{\br_L^{(k)}}_{\rL_k^2(\Omega, W)}^2
        &\leq \sum_{\ell=0}^L M_{\ell}^{-1}  \bnorm{{\bp_\ell^{(k)} - \EE_k[\bp_\ell^{(k)}]}}^2_{\rL_k^2(\Omega, W)}
        + \bnorm{\EE_k[\bq_{L}^{(k)} - \bq^{(k)}]}_W^2 \\
        &\quad+ \lambda^2 \bnorm{\bz_L^{(k)} - \bz^{(k)}}_{\rL^2_k(\Omega, W)}^2 \,.
    \end{align*}
    Considering the left side of assumption~\eqref{eq:assumption-alpha-z}
    and applying Jensen's inequality gives
    \begin{align*}
        \tfrac{1}{2} \bnorm{\br_L^{(k)}}_{\rL_k^2(\Omega, W)}^2
        \leq \sum_{\ell=0}^L M_{\ell}^{-1}  \bnorm{{\bp_\ell^{(k)} - \EE_k[\bp_\ell^{(k)}]}}^2_{\rL_k^2(\Omega, W)}
        \! + (\lambda^2 c_{\cG} + 1) \, \EE_k \big[\bnorm{\bq_L^{(k)} - \bq^{(k)}}_{W}^2 \big]
    \end{align*}
    Thus, with~\eqref{eq:assumption-beta-q}
    and~\eqref{eq:assumption-alpha-z} both error
    contributions are controlled and similar as 
    in~\cite{cliffe2011multilevel, giles2008multilevel},
    we can achieve with some bias-variance trade-off $\theta \in (0, 1)$
    \begin{align*}
        \sum_{\ell=0}^L M_{\ell}^{-1}  \bnorm{{\bp_\ell^{(k)} - \EE_k[\bp_\ell^{(k)}]}}^2_{\rL_k^2(\Omega, W)}
        &\leq \theta \epsilon_k^2
        \quad \text{and} \quad
        \EE_k \big[\bnorm{\bq_L^{(k)} - \bq^{(k)}}_{W}^2 \big]
        \leq \frac{(1 - \theta)}{(\lambda^2 c_{\cG} + 1)} \, \epsilon_k^2
    \end{align*}
    through an appropriate choice of $\tset{M_\ell^{(k)}}_{\ell=0}^L$
    such that~\eqref{eq:gradient-complexity} holds
    following \eqref{eq:mlmc-complexity}.
    Details on this construction are given in~\cite{cliffe2011multilevel}
    in the appended proof on the generalized MLMC theorem and
    in~\Cref{alg:bmlsgd} which presents a way on how to find
    $\tset{M_\ell^{(k)}}_{\ell=0}^L$ algorithmically.
\end{proof}

\begin{theorem}[Convergence and $\epsilon$-Cost of MLSGD]
    \label{thm:convergence-mlsgd}
    Suppose $\bz^* \in {Z}$ denotes
    a solution to~\Cref{problem:ocp} satisfying~\eqref{eq:variational-inequality}.
    There exists a sequence of multilevel batches
    $\bset{\tset{M_{k,\ell}}_{\ell=0}^{L}}_{k=0}^{K-1}$,
    step sizes $\tset{t_k}_{k=0}^{K-1} \subset \RR_{\geq 0}$,
    and some $\rho \in (0, 1)$, such that
    \begin{equation}
        \label{eq:convergence-mlsgd}
        e_K \coloneqq \bnorm{\bz^{(K)}_L - \bz^*}_{\rL^2(\Omega, W)}
        \quad \text{converges linearly, i.e.,} \quad
        e_K \leq \cO \big( \rho^{K} \big).
    \end{equation}
    Further, reaching an error of $e_K < \epsilon$ smaller than accuracy $\epsilon > 0$ comes at the cost
    \begin{equation}
        \label{eq:cost-mlsgd}
        \rC_{\epsilon} \lesssim
        \begin{cases}
            \epsilon^{-2} & \beta > \gamma, \\
            \epsilon^{-2} (\log(\epsilon))^2 & \beta = \gamma, \\
            \epsilon^{-2 -(\gamma-\beta)/\alpha} & \beta < \gamma.
        \end{cases}
    \end{equation}
\end{theorem}

\begin{proof}
    The proof is organized in two steps,
    first we show the linear convergence
    and then we estimate the cost of the algorithm.
    We define for the further proof
    \begin{align*}
        e_{k+1} \coloneqq \bnorm{\bz^{(k+1)}_L - \bz^*}_{\rL_k^2(\Omega, W)}, \quad
        e_k \coloneqq \bnorm{\bz^{(k)}_L - \bz^*}_{\rL_k^2(\Omega, W)}.
    \end{align*}

    \smallskip

    \paragraph{(i)} Adapting the arguments of~\cite{geiersbach2020stochastic}
    to the iteration scheme~\eqref{eq:mlsgd-iteration} gives
    \begin{align*}
        e_{k+1}^2 \!
        &= \bnorm{\pi_{Z}
            \big(\bz^{(k)}_L - t_k E^{\text{ML}}[\bg_L^{(k)}]\big)
            - \pi_{Z}\left(\textcolor{black}{\bz^*} - t_k \nabla J(\bz^*) \right)}_{\rL_k^2(\Omega, W)}^2 \\
        &\leq \bnorm{\bz^{(k)}_L-t_k E^{\text{ML}}[\bg_L^{(k)}]-\bz^* {+ t_k \nabla J(\bz^*)}}_{\rL_k^2(\Omega, W)}^2\\
        &= e_k^2
        + t_k^2\bnorm{E^{\text{ML}}[\bg_L^{(k)}] - \nabla J(\bz^*)}_{\rL_k^2(\Omega, W)}^2 \\
        &\quad - 2 t_k \bsprod{\bz^{(k)}_L \!\! - \bz^* \!,  E^{\text{ML}}[\bg_L^{(k)}] {- \nabla J(\bz^*)} }_{\rL_k^2(\Omega, W)} \,.
    \end{align*}
    By~\eqref{eq:error-gradient-estimation},
    it is $E^{\text{ML}}[\bg_L^{(k)}] = \br^{(k)}_L + \nabla J(\bz^{(k)})$ and we can write
    \begin{align*}
        e_{k+1}^2
        &\leq e_k^2
        -  2 t_k  \bsprod{\bz^{(k)}_L -\bz^*,  \br^{(k)}_L +  \nabla J(\bz^{(k)}) { - \nabla J(\bz^*)} }_{\rL_k^2(\Omega, W)}\\
        &\quad+  t_k^2  \bnorm{\br^{(k)}_L +  \nabla J(\bz^{(k)}){ - \nabla J(\bz^*)}}_{\rL_k^2(\Omega, W)}^2.
    \end{align*}
    Since $J$ is $\mu$-strongly convex~\eqref{eq:strongly-convex}
    and by adding and subtracting $\nabla J (\bz_L^{(k)})$,
    it follows
    \begin{align*}
        -\bsprod{\bz^{(k)}_L - \bz^*, \nabla J(\bz^{(k)}) { - \nabla J(\bz^*)}}_{\rL_k^2(\Omega, W)}
        &\leq -\mu e_k^2 + e_k \bnorm{\nabla J(\bz^{(k)}) - \nabla J (\bz_L^{(k)})}_{\rL_k^2(\Omega, W)}, \\
        - \bsprod{\bz^{(k)}_L -\bz^*,  \br^{(k)}_L }_{\rL_k^2(\Omega, W)}
        &\leq e_k \, \bnorm{\br^{(k)}_L }_{\rL_k^2(\Omega, W)}.
    \end{align*}
    Therefore, we can estimate
    \begin{align*}
        e_{k+1}^2 \leq \, e_k^2
        &+ 2 t_k \Big(e_k \, \bnorm{\br^{(k)}_L }_{\rL_k^2(\Omega, W)}
        + e_k \bnorm{\nabla J(\bz^{(k)}) - \nabla J (\bz_L^{(k)})}_{\rL_k^2(\Omega, W)} -\mu e_k^2 \Big) \\
        &+ 2 t_k^2 \big(\bnorm{\br_L^{(k)}}_{\rL_k^2(\Omega, W)}^2 + \bnorm{\nabla J(\bz^{(k)}) { - \nabla J(\bz^*)}}_{\rL_k^2(\Omega, W)}^2 \big) \,.
    \end{align*}
    Further, by~\eqref{eq:lipschitz-continuity-gradient}
    we see that
    \begin{align*}
        \bnorm{\nabla J(\bz^{(k)}) - \nabla J(\bz^{*})}_{\rL_k^2(\Omega, W)}^2
        &\leq c_{\mathrm{Lip}}^2 \bnorm{\bz^{(k)} - \bz^*}_{\rL_k^2(\Omega, W)}^2 \\
        &\leq 2 c_{\mathrm{Lip}}^2 \big(e_k^2 + \bnorm{\bz^{(k)} - \bz_L^{(k)}}_{\rL_k^2(\Omega, W)}^2 \big) \,.
    \end{align*}
    and since the true gradient in step $k$ is given by $\nabla J(\bz^{(k)}) = \lambda \bz^{(k)} - \EE_k[\bq^{(k)}]$,
    it is
    \begin{align*}
        \bnorm{\nabla J(\bz^{(k)}) - \nabla J (\bz_L^{(k)}))}_{\rL_k^2(\Omega, W)}
        &\leq \bnorm{\EE_k[\bq^{(k)} - \bq^{(k)}_L]}_{\rL_k^2(\Omega, W)}
        + \lambda \bnorm{\bz^{(k)} - \bz_L^{(k)}}_{\rL_k^2(\Omega, W)} \,.
    \end{align*}
    In conclusion, the estimate for $e_{k+1}^2$ is given by
    \begin{align*}
        e_{k+1}^2 \leq \, e_k^2
        &+ 2 t_k e_k \bnorm{\br^{(k)}_L}_{\rL_k^2(\Omega, W)}
        + 2 t_k e_k \bnorm{\EE_k[\bq^{(k)} - \bq^{(k)}_L]}_{\rL_k^2(\Omega, W)} \\
        &+ 2 t_k e_k \lambda \bnorm{\bz^{(k)} - \bz_L^{(k)}}_{\rL_k^2(\Omega, W)}
        + 2 t_k^2 \bnorm{\br^{(k)}_L}_{\rL_k^2(\Omega, W)}^2 \\
        &+ 4 t_k^2 c_{\mathrm{Lip}}^2 \bnorm{\bz^{(k)} - \bz_L^{(k)}}_{\rL_k^2(\Omega, W)}^2
        + 4 t_k^2 c_{\mathrm{Lip}}^2 e_k^2
        - 2 t_k \mu e_k^2 \,.
    \end{align*}
    Now, we apply~\Cref{lem:error-gradient-estimation} by choosing
    $\epsilon_k$ in every iteration $k$ as $\eta \cdot e_k$
    with a sufficiently small $\eta>0$.
    Then, there exists a sequence of multilevel
    batches $\bset{\tset{M_\ell^{(k)}}_{\ell=0}^{L}}_{k \in \NN}$,
    such that
    \begin{equation}
        \label{eq:quadratic-minimization}
        e_{k+1}^2 \leq \, e_k^2
        \big( (4 c_{\mathrm{Lip}}^2 (1+\eta^2 {c_{\lambda\cG}}) + {4}\eta^2) t_k^2
        + 2 (\sqrt{2}\eta +\eta + \sqrt{{c_{\lambda\cG}} }\lambda\eta - \mu) t_k + 1 \big)
        \eqqcolon e_k^2 \rho_k^2
    \end{equation}
    where we used $(1-\theta)\leq1$ and $c_{\lambda\cG} \coloneqq \tfrac{c_{\cG}}{\lambda^2 c_\cG + 1}$.
    By taking the square root of the above expression, it follows by induction
    that the error converges linearly if $\rho_k < 1$ in each step.
    This can be achieved by choosing the step size $t_k$ and the factor $\eta$,
    such that the quadratic expression of $\rho_k$ in~\eqref{eq:quadratic-minimization}
    is minimized and $t_k>0$ holds.
    As a result, by taking the expectation and the tower property,
    the total error in the final step $K$ is given by
    $\cO(\rho^K)$, with $\rho \coloneq \sup \tset{\rho_k \colon k = 1, \ldots, K}$.

    \paragraph{(ii)} We denote with $\rC_k$ again the cost of the $k$-th descent.
    By~\Cref{lem:error-gradient-estimation},
    the summed cost $\rC_\epsilon$ to achieve a total
    error of $e_K < \epsilon$ is given by
    \begin{align*}
        \rC_{\epsilon} &= \sum_{k=0}^K \rC_k \lesssim \sum_{k=0}^K
        \begin{cases}
            \epsilon_k^{-2} & \beta > \gamma, \\
            \epsilon_k^{-2} \log(\epsilon_k)^2 & \beta = \gamma, \\
            \epsilon_k^{-2-(\gamma-\beta)/\alpha} & \beta < \gamma.
        \end{cases}
    \end{align*}
    Since we chose the multilevel batch such that $\eta^{-1}\,\epsilon_k$ equals the
    error in each step, and since $e_k =\cO(\rho^k)$,
    we get similar to~\cite{van2019robust} for large $K \gg 1$ the asymptotic behavior
    \begin{align*}
        \rC_{\epsilon} \lesssim \sum_{k=0}^K
        \begin{cases}
            \rho^{-2k}   \\
            \rho^{-2k} (k \log(\rho))^2 \\
            \rho^{-2k - (\gamma-\beta)k/\alpha}
        \end{cases}
        \hspace{-0.3cm} \sim
        \begin{cases}
            \rho^{-2K-2} \\
            K^2 \, \rho^{-2K-2} \\
            \rho^{-2(K+1) - (\gamma-\beta)(K+1)/\alpha} \\
        \end{cases}
        \hspace{-0.5cm}\sim
        \begin{cases}
            \epsilon^{-2} & \beta > \gamma, \\
            \epsilon^{-2} (\log(\epsilon))^2 & \beta = \gamma, \\
            \epsilon^{-2 -(\gamma-\beta)/\alpha} & \beta < \gamma,
        \end{cases}
    \end{align*}
    using the geometric series and $\epsilon = \cO(\rho^K)$.
\end{proof}

\begin{remark}
    \begin{enumerate}
        \item Even though MLSGD improves the rate from $\cO(K^{-1/2})$
        to linear convergence $\cO(\rho^K)$, as discussed in~\Cref{subsec:bsgd-discussion},
        similar results can be achieved
        with large enough batches and gradient aggregation methods.
        However, as pointed out in~\cite{bottou2018optimization}
        and observed in~\Cref{subsec:bsgd-experiments},
        this comes with an increased computational cost.
        Our approach, leveraging the MLMC method, reduces the cost
        of the gradient estimation dramatically,
        as demonstrated numerically in~\Cref{subsec:mlsgd-experiments}
        and now proven by~\Cref{thm:convergence-mlsgd}.
        \item The parameter $\eta>0$ has to be chosen small enough
        that minimizing~\eqref{eq:quadratic-minimization} is achieved with $t_k>0$.
        Since $\eta~<~\frac{\mu}{\sqrt{2} +1 + \sqrt{{c_{\lambda\cG}}}\lambda}$ is independent of $k$,
        this suggests the usage of an adaptive algorithm that successively enlarges the multilevel batch.
        In particular, as the algorithm runs and $e_k$ decreases,
        this leads with $\epsilon_k = \eta \cdot e_k$ to a decreasing target error $\epsilon_k$
        and thus to an increased computational cost of each iteration.
    \end{enumerate}
\end{remark}

The following corollary motivates
implementing the targeted multilevel error
by $\epsilon_k = \eta \cdot \Vert E^{\text{ML}} [\bg_L^{(k)}] \Vert_W$.
This idea is similar to~\cite{van2019robust} and later used in \Cref{alg:bmlsgd}.

\begin{corollary}
    \label{cor:convergence-mlsgd}
    {Suppose $\bz^* \in {Z}$ denotes
    a solution to~\Cref{problem:ocp} satisfying~\eqref{eq:optimality-condition}}
    {a.e.~$\bx\in\cD$}.
    There exists a sequence of multilevel batches
    $\bset{\tset{M_\ell^{(k)}}_{\ell=0}^{L}}_{k=0}^{K-1}$,
    step sizes $\tset{t_k}_{k=0}^{K-1} \subset \RR_{\geq 0}$,
    and some $\rho \in (0, 1)$, such that
    \begin{equation*}
        \EE \big[J(\bz^{(K)}_L) - J(\bz^*)\big] \leq \cO (\rho^K) \quad\text{ and }\quad \bnorm{E^{\rM\rL}[\bg_L^{(K)}]}_{\rL^2(\Omega, W)} \leq \cO(\rho^K) \,.
    \end{equation*}
\end{corollary}
\begin{proof}
    The first expression follows from the strong
    $\mu$-convexity~\eqref{eq:strongly-convex-equivalent}
    together with~\eqref{eq:lipschitz-continuity-gradient}
    \begin{align*}
        J(\bz^{(K)}_L)  -  J(\bz^*) &\leq
        \bsprod{\nabla J[\bz^{(K)}_L], \bz^{(K)}_L  - \bz^*}_{W}
         - \tfrac{\mu}{2} \bnorm{\bz^{(K)}_L  - \bz^*}_{W}^2 \\
        &\leq (c_{\mathrm{Lip}}  -  \tfrac{\mu}{2}) \bnorm{\bz^{(K)}_L  - \bz^*}_{W}^2
    \end{align*}
    and by taking the expectation on both sides.
    The second expression follows from
    bounding $\bnorm{\nabla J(\bz^{(k)})}_{\rL^2_k(\Omega, W)}$
    and $\bnorm{\br_L^{(k)}}_{\rL_k^2(\Omega, W)}$
    in~\eqref{eq:error-gradient-estimation}
    by extending with the fact that \eqref{eq:optimality-condition} holds and with
    similar arguments as in~\Cref{thm:convergence-mlsgd}.
    Taking the expectation over all possible optimization
    trajectories concludes the result.
\end{proof}

    \section{Adaptivity and Budgeting}\label{sec:adaptivity-and-budgeting}

Having established the MLSGD method,
particularly its convergence in~\Cref{thm:convergence-mlsgd}
and its functionality in~\Cref{alg:mlsgd},
we now address the question of how to choose the multilevel
batches $\bset{\tset{M_{k,\ell}}_{\ell=0}^{L}}_{k=0}^{K-1}$
and step sizes $\tset{t_k}_{k=0}^{K-1}$
given computational constraints
such as memory limitations $\mathrm{Mem}_{0}$
and CPU-time budgets $\abs{\cP} \cdot \rT_{0}$,
$\rT_{0}$ being the reserved time budget on the cluster
and $\abs{\cP}$ being the cardinality of the set of processing units.
Building upon~\cite{baumgarten2024fully, baumgarten2025budgeted},
we formulate the following knapsack problem,
which is then solved with the budgeted
multilevel stochastic gradient descent (BMLSGD) method,
adaptively determining the multilevel batches and step sizes.

\begin{problem}[MLSGD Knapsack]
    \label{problem:knapsack-ocp}
    Suppose the solution $\bz^* \in {Z}$ to the optimal control problem~\eqref{eq:ocp-objective} is given,
    find the optimal sequence of step sizes $\tset{t_k}_{k=0}^{K-1}$
    and multilevel batches $\bset{\tset{M_{k,\ell}}_{\ell=0}^{L_k}}_{k=0}^{K-1}$,
    such that the total error $e_K$ is minimized,
    while staying within a CPU-time $\abs{\cP} \cdot \rT_{0}$ and memory ${\mathrm{Mem}}_{0}$ budget
    \begin{subequations}
        \label{eq:knapsack-ocp}
        \begin{align}
            \label{eq:knapsack-ocp-error}
            \hspace{-0.5cm} \min_{\,\,\,\,\,\set{t_k, \tset{M_{k,\ell}}_{\ell=0}^{L_k}}_{k=0}^K}
            \bnorm{\bz_L^{(K)} - \bz^*&}_{\rL^2(\Omega,W)} \\[1mm]
            \label{eq:knapsack-ocp-ct-constraints}
            \sum_{k=0}^{K_{\phantom{,}}\!\!-1}
            \sum_{\ell = 0}^{L_{k\phantom{,}}}
            \sum_{m=1}^{M_{k,\ell}} \rC^{\rC\rT} \big(({\bv}_\ell^{(m,k)}\!, \bp_\ell^{(m,k)}) \big)
            &\leq \abs{\cP} \cdot \rT_{0} \\
            \label{eq:knapsack-ocp-mem-constraints}
            \rC^{\mathrm{Mem}}\big(({\bv}_L^{(K-1)}\!, \bp_L^{(K-1)})\big)
            &< {\mathrm{Mem}}_{0} .
        \end{align}
    \end{subequations}
\end{problem}

\Cref{problem:knapsack-ocp} is an NP-hard integer resource allocation problem
which we solve with distributed dynamic programming techniques
utilizing an optimal policy to find the multilevel batches in
each iteration of the optimization algorithm.
To respect the cost constraints~\eqref{eq:knapsack-ocp-ct-constraints} and~\eqref{eq:knapsack-ocp-mem-constraints},
we utilize assumptions~\eqref{eq:assumption-gamma-CT} and~\eqref{eq:assumption-gamma-mem}.
Before we present the BMLSGD method in~\Cref{alg:bmlsgd},
we can conclude with similar arguments as
in~\cite{baumgarten2024fully, baumgarten2025budgeted}
that the error of every feasible solution is bounded
from below by the memory constraint
and from above by the CPU-time budget.

\begin{corollary}[Lower and Upper bound of MLSGD]
    \label{cor:upper-and-lower-bound}
    The minimum~\eqref{eq:knapsack-ocp-error} is bounded by
    the imposed constraints~\eqref{eq:knapsack-ocp-ct-constraints}
    and~\eqref{eq:knapsack-ocp-mem-constraints} as
    \begin{equation}
        \label{eq:upper-and-lower-bound}
        c_{\mathrm{Mem}}{\mathrm{Mem}}_{0}^{-\alpha/\gamma_{\mathrm{Mem}}}
        < \,\, \bnorm{\bz_L^{(K)} - \bz^*}_W \,\,  \lesssim \,\,
        (1 - \lambda_{\mathrm{p}}) \rT_{0}^{-\delta} + \lambda_{\mathrm{p}} (\abs{\cP} \cdot \rT_{0})^{-\delta}.
    \end{equation}
    Here, $\delta = \min \bset{\tfrac12, \frac{\alpha}{2 \alpha + (\gamma - \beta)}}$
    is the convergence rate with respect to the computational resources and
    $\lambda_{\mathrm{p}} \in [0, 1)$ is the parallelizable percentage of the code
    and $c_{\mathrm{Mem}} > 0$ is some memory cost constant.
\end{corollary}

\begin{proof}
    Following~\cite{baumgarten2024fully, gustafson1988reevaluating},
    we state revised arguments for the right-hand side.
    Let $\rC_\epsilon^{\mathrm{ref}}$ denote the serial execution time required to reach accuracy $\epsilon$.
    Decomposing it into a parallelizable $\lambda_{\rp} \in [0, 1)$ fraction
    and a non-parallelizable fraction $(1 - \lambda_{\rp})$, we write
    $\rC_\epsilon^{\mathrm{ref}} = (1 - \lambda_{\rp}) \rC_\epsilon^{\mathrm{ref}} + \lambda_{\rp} \rC_\epsilon^{\mathrm{ref}}$.
    On parallel hardware, the execution time reduces to
    $\rC_\epsilon^{\mathrm{par}} = (1 - \lambda_{\rp}) \rC_\epsilon^{\mathrm{ref}} + \lambda_{\rp} f(\abs{\cP}) \rC_\epsilon^{\mathrm{ref}}$,
    where $f \colon \RR_{\geq 0} \to (0,1]$ is monotone decreasing.
    The resulting reduction factor is
    $\rR = \rC_\epsilon^{\mathrm{par}} / \rC_\epsilon^{\mathrm{ref}} = (1 - \lambda_{\rp}) + \lambda_{\rp} f(\abs{\cP}) < 1$.
    Now fix a time budget $\rT_0$.
    Inverting~\eqref{eq:cost-mlsgd} yields $\epsilon_{\rs} \lesssim \rT_0^{-\delta}$ on serial hardware.
    For parallel hardware, the reduction factor $\rR$ gives
    $\epsilon_{\rp} \lesssim \rR \cdot \rT_0^{-\delta} = (1 - \lambda_{\rp}) \rT_0^{-\delta} + \lambda_{\rp} f(\abs{\cP}) \rT_0^{-\delta}$.
    The stated bound follows if $f(|\cP|)=|\cP|^{-\delta}$,
    which by~\eqref{eq:cost-mlsgd} holds when each processing
    unit handles a sufficiently large share of the total workload.

    For the left-hand side, we know by using the multiindex algorithm of~\cite{baumgarten2025budgeted}
    that the memory footprint is dominated by the largest level $L$
    and that in the limit of infinite computing time the only error component
    remaining is the bias $h_L^\alpha \sim \norm{\EE[\bq_L - \bq]}_W \lesssim \tnorm{\bz_L^{(K)} - \bz^*}_W$
    by~\eqref{eq:assumption-alpha-q}.
    To enforce~\eqref{eq:knapsack-ocp-mem-constraints},
    we have to ensure that $m_k h_L^{-\gamma_{\mathrm{Mem}}} < \mathrm{Mem}_0$
    using assumption~\eqref{eq:assumption-gamma-mem}.
    Thus, $h_L > (m_k / \mathrm{Mem}_0)^{1/\gamma_{\mathrm{Mem}}}$, which implies
    $c_{\mathrm{Mem}} \mathrm{Mem}_0^{-\alpha/\gamma_{\mathrm{Mem}}} < \tnorm{\bz_L^{(K)} - \bz^*}_W$.
\end{proof}

    \subsection{Algorithm}
\label{subsec:adaptivity-and-budgeting-algorithm}

Building upon~\Cref{alg:bsgd} and~\Cref{alg:mlsgd},
we outline again the method as functional pseudocode.
Since \Cref{problem:knapsack-ocp} is intractable,
we use the estimators given in~\Cref{sec:multilevel-monte-carlo}
in combination with the theory from~\Cref{subsec:mlsgd-analysis}.
We define the global variables for the leftover time budget $\rT_k$
and the total available memory $\mathrm{Mem}_k$
(we store the history of these values along the optimization steps $k$)
and initialize them with the given $\rT_0$
and $\mathrm{Mem}_0$ from~\Cref{problem:knapsack-ocp}.
We further set the constants for the error target reduction $\eta \in (0,1]$
and the bias-variance tradeoff $\theta \in (0,1)$ to start the algorithm.

\smallskip

\paragraph{The Init function} Starting the call stack in
the \texttt{Init} function with an initial guess on $\bz_{L}^{(0)}$,
an initial multilevel batch $\tset{M_{0, \ell}}_{\ell=0}^{L_0}$,
and the initial step size $t_0$,
the \texttt{MultiLevelEstimation} function
and the \texttt{GradientDescent} function from~\Cref{alg:mlsgd} are called.
This gives the initial estimates on the adjoint $E^{\text{ML}}[\bq_{L}^{(0)}]$,
on the initial objective $J^{\text{ML}}(\bz_{L}^{(0)})$
and with that, the gradient $E^{\text{ML}}[\bg_{L}^{(0)}]$
and the first updated control $\bz_{L}^{(1)}$.
The computational cost is dominated by the \texttt{MultiLevelEstimation} function,
thus after its first call, the time budget and the leftover memory are updated with
$\rT_{1} \leftarrow \rT_{0} - \rC^{\rC\rT}_0$ and $\mathrm{Mem}_{1} \leftarrow \mathrm{Mem}_{0} - \rC^{\mathrm{Mem}}_0$,
where the cost measurements $\rC^{\rC\rT}_k$ and $\rC^{\mathrm{Mem}}_k$ are given by
\begin{equation}
    \label{eq:cost-measurements}
    \rC^{\rC\rT}_k
    = \sum_{\ell=0}^{\,L_{k\phantom{,}}} \sum_{m=1}^{M_{k,\ell}} \rC^{\rC\rT} \big(({\bv}_\ell^{(m,k)}, \bp_\ell^{(m,k)}) \big)
    \qquad \text{and} \qquad
    \rC^{\mathrm{Mem}}_k
    \sim m_k h_{L_k}^{-\gamma_{\mathrm{Mem}}}.
\end{equation}
We refer to~\cite{baumgarten2025budgeted} for a more detailed discussion on the time cost
measurement in parallel and why the memory cost can be bounded by the highest level $L_k$
following the assumption~\eqref{eq:assumption-gamma-mem}.
Subsequently, we call the \texttt{BMLSGD} function
with the new control $\bz_{L}^{(1)}$
and the first error target
$\epsilon_1 \leftarrow \eta \cdot \bnorm{E^{\text{ML}}[\bg_{L}^{(0)}]}_{W}$
(cf.~\cite{baumgarten2024fully} for a discussion why $\eta$
is relevant for the parallel efficiency).

\smallskip

\paragraph{The BMLSGD function}
The \texttt{BMLSGD} function in~\Cref{alg:bmlsgd}
acts as the main driver of the algorithm by recursively
minimizing~\eqref{eq:knapsack-ocp-error}, while
enforcing the computational constraints~\eqref{eq:knapsack-ocp-ct-constraints}
and~\eqref{eq:knapsack-ocp-mem-constraints}.
This involves checking the leftover time budget $\rT_k$ in a guard clause of the function,
and adapting the largest level $L_k$ if the FE error, estimated with~\eqref{eq:squared-bias-estimate-mlmc},
is larger than a predefined fraction of the target error $\theta \,\epsilon_k$.
To find the optimal multilevel batch $\tset{M_{k, \ell}^{\text{opt}}}_{\ell=0}^{L_k}$ in each iteration $k$,
we apply techniques introduced
in~\cite{giles2015multilevel, baumgarten2025budgeted}
utilizing estimates from the previous optimization step $k - 1$.
Particularly, we use
\begin{equation}
    \label{eq:optimal-Ml-estimated}
    M_{k, \ell}^{\text{opt}}
    = \ceil {
        \roundlr{\sqrt{\theta} \epsilon_{k}}^{-2}\!\!
        \sqrt{\frac{s_\ell^2[\bp_\ell^{(k-1)}]}{(M_{k - 1, \ell} - 1) \, \widehat{\rC}^{\rC\rT}_{k - 1, \ell}}}
        \roundlr{
            \sum_{\ell'=0}^{L_{k}}
            \sqrt{\frac{s_{\ell'}^2[\bp_{\ell'}^{(k-1)}] \, \widehat{\rC}^{\rC\rT}_{k-1, \ell'}}{M_{k - 1, \ell'} - 1}} \,\,
        }
    }
\end{equation}
to get the optimal number of samples $M_{k, \ell}^{\text{opt}}$ on each level $\ell$
where $s_\ell^2[\bp_\ell^{(k-1)}]$ is the second order sum from~\eqref{eq:sampling-error-estimator}.
The estimate $\widehat{\rC}^{\rC\rT}_{k - 1, \ell}$ is the averaged cost of the previous
iteration $k - 1$ on level $\ell$, thus with the sample size $M_{k - 1, \ell}$, it is
\begin{align*}
    \widehat{\rC}^{\rC\rT}_{k - 1, \ell}
    = \frac{1}{M_{k - 1, \ell}} \sum_{m=1}^{M_{k - 1, \ell}}
    \rC^{\rC\rT} \big(({\bv}_{\ell}^{(m,k-1)}, \bp_{\ell}^{(m,k-1)}) \big).
\end{align*}
Once the new multilevel batch $\tset{M_{k, \ell}^{\text{opt}}}_{\ell=0}^{L_k}$ is chosen,
the \texttt{NotFeasible} function is called.
This function stops the computation if the new batch
is expected to take longer than the leftover time budget $\rT_k$ or,
if, by appending a new level, the memory constraint is violated.
Else, we call the \texttt{MultiLevelEstimation} function to compute new estimates on the
adjoint $E^{\text{ML}}[\bq_{L}^{(k)}]$ and the objective $J^{\text{ML}}(\bz_{L}^{(k)})$
using the new multilevel batch $\tset{M_{k, \ell}^{\text{opt}}}_{\ell=0}^{L_k}$.
We then determine the new control $\bz_{L}^{(k+1)}$ by calling the
\texttt{AdaptiveGradientDescent} function, update the budgets with~\eqref{eq:cost-measurements},
and start a new optimization step with the
\texttt{BMLSGD} function taking the new control $\bz_{L}^{(k+1)}$,
and the new error target
$\epsilon_{k+1} \leftarrow \eta \cdot \bnorm{E^{\text{ML}}[\bg_{L}^{(k)}]}_{W}$
as input.

\smallskip

\paragraph{Feasibility check}
Within the \texttt{NotFeasible} function, the upcoming computational cost is estimated
and checked against the remaining budgets:
\begin{align*}
    \label{eq:cost-constraints}
    \underbrace{\sum_{\ell=0}^{L_k} M_{k,\ell}^{\text{opt}} \widehat{\rC}^{\rC\rT}_{k - 1, \ell} > \rT_{k}}_{\text{new batch is too expensive in CPU-time}}
    \quad \text{or} \quad
    \underbrace{m_k h_{L_k}^{-\gamma_{\mathrm{Mem}}} > {\mathrm{Mem}}_{k}}_{\text{new batch is too expensive in memory}}
\end{align*}
We recall that $\mathrm{Mem}_{k}$ is the leftover memory, i.e.~the difference between the
initial memory budget ${\mathrm{Mem}}_{0}$ and the memory footprint of
the last computational result.
In either case, if the new multilevel batch is too expensive
in CPU-time or memory, the algorithm stops and returns the current control
as the solution to~\Cref{problem:knapsack-ocp}.

\smallskip

\paragraph{Adaptive Gradient Descent}
By \Cref{thm:convergence-mlsgd}, we do have to determine besides the optimal multilevel
batch $\tset{M_{k, \ell}^{\text{opt}}}_{\ell=0}^{L_k}$ also the step size $t_k$ in
order to minimize~\eqref{eq:quadratic-minimization}.
To this end, we use the adaptive step size techniques
introduced in~\cite{koehne2024adaptivestepsizespreconditioned}
in combination with the already computed MLMC estimates
and the gradient $E^{\text{ML}}[\bg_{L}^{(k)}]$
used in~\eqref{eq:mlsgd-iteration}.
Particularly, we compute
\begin{equation}
    \label{eq:adaptive-step-size}
    t_{k} =
    \frac{\tnorm{E^{\text{ML}}[\bg_{L}^{(k)}]}_{W}^2 - \widehat{\err}^{\text{sam}}_k}
    {\widehat{c}_{\mathrm{Lip}} \tnorm{E^{\text{ML}}[\bg_{L}^{(k)}]}_{W}^2}
    \quad \text{with} \quad
    \widehat{c}_\mathrm{Lip} =
    \frac{\tnorm{E^{\text{ML}}[\bg_{L}^{(k)}]-E^{\text{ML}}[\bg_{L}^{(k-1)}]}_{W}}
    {\tnorm{t_{k-1}E^{\text{ML}}[\bg_{L}^{(k-1)}]}_{W}},
\end{equation}
where $\widehat{\err}_{\text{sam}}$ is determined with~\eqref{eq:sampling-error-estimator},
and $\widehat{c}_\mathrm{Lip}$ is an estimate of the Lipschitz
constant following~\eqref{eq:lipschitz-continuity-gradient}
which takes the previous gradient $E^{\text{ML}}[\bg_{L}^{(k-1)}]$ from memory.
With the new step size $t_k$, we finally compute the new control
$\bz_{L}^{(k+1)}$ following the update rule in~\eqref{eq:mlsgd-iteration}
and return to the \texttt{BMLSGD} function again.

\smallskip

\begin{algorithm}
    \caption{Budgeted Multilevel Stochastic Gradient Descent (BMLSGD)}
    \label{alg:bmlsgd}
    \begin{align*}
        &\texttt{global const } (\eta, \, \theta)^\top \leftarrow (0.9, 0.5)^\top
        \hspace{6.2mm} \text{// Tested choice for both values} \\[1mm]
        &\texttt{global } (\rT_k, \, \mathrm{Mem}_k)^\top \hspace{1mm} \leftarrow (\rT_0, \, \mathrm{Mem}_0)^\top
        \hspace{1mm} \text{// Initialize with time and memory budget} \\[5mm]
        &\texttt{function Init}(\bz_{L}^{(0)} \!, \tset{M_{0, \ell}}_{\ell=0}^{L}, t_0) \colon
        \hspace{0.1cm} \text{// Calls functions from~\Cref{alg:mlsgd}} \\[-1mm]
        &\quad
        \begin{cases}
            E^{\text{ML}}[\bq_{L}^{(0)}], \, J^{\text{ML}}(\bz_{L}^{(0)})
            \hspace{0.6cm} \leftarrow &\hspace{0.3cm} \texttt{MultiLevelEstimation}(\bz_{L}^{(0)} \!, \tset{M_{0, \ell}}_{\ell=0}^{L}) \\[1mm]
            \hspace{1.0cm} \bz_{L}^{(1)} \!, \, E^{\text{ML}}[\bg_{L}^{(0)}]
            \hspace{0.6cm} \leftarrow &\hspace{0.3cm} \texttt{GradientDescent}(\bz_{L}^{(0)} \!, E^{\text{ML}}[\bq_{L}^{(0)}], t_0) \\[1mm]
            \hspace{1.08cm} (\rT_1, \, \mathrm{Mem}_1)^\top\hspace{0.68cm}
            \leftarrow &\hspace{0.3cm} (\rT_{0} - \rC_0^{\mathrm{CT}} \!, \, \mathrm{Mem}_0 - \rC^\mathrm{Mem}_0)^\top \\[1mm]
            \texttt{return }
            &\hspace{0.3cm} \texttt{BMLSGD}(\bz_{L}^{(1)} \!, \, \eta \cdot \bnorm{E^{\text{ML}}[\bg_{L}^{(0)}]}_{W})
        \end{cases} \\ \\[-1mm]
        &\texttt{function BMLSGD}(\bz_{L}^{(k)} \!, \, \epsilon_k) \colon \\[-1mm]
        &\quad
        \begin{cases}
            \texttt{if } \rT_{k} < 0.05 \, \rT_{0} \colon
            &\texttt{return } \bz_L^{(k)} \\[1mm]
            \texttt{if } \widehat{\err}^{\text{num}}_{k-1} \geq (1 - \theta) \, \epsilon_{k}^2 \colon
            &{L_{k}} \leftarrow {L_{k-1}} + 1 \\[1mm]
            \texttt{for } \ell=0,\dots, L_k \colon
            &M_{k, \ell}^{\text{opt}} \leftarrow \eqref{eq:optimal-Ml-estimated} \\[1mm]
            \texttt{if } \texttt{NotFeasible}(\tset{M_{k, \ell}^{\text{opt}}}_{\ell=0}^{L_k}) \colon
            \hspace{-0.2cm}&\texttt{return } \bz_L^{(k)} \\[4mm]
            E^{\text{ML}}[\bq_{L}^{(k)}], \, J^{\text{ML}}(\bz_{L}^{(k)})
            \hspace{0.55cm} \leftarrow &\texttt{MultiLevelEstimation}(\bz_{L}^{(k)} \!, \tset{M_{k, \ell}^{\text{opt}}}_{\ell=0}^{L_k}) \\[1mm]
            \hspace{0.65cm} \bz_{L}^{(k+1)} \!, \, E^{\text{ML}}[\bg_{L}^{(k)}]
            \hspace{0.55cm} \leftarrow &\texttt{AdaptiveGradientDescent}(\bz_{L}^{(k)} \!, E^{\text{ML}}[\bq_{L}^{(k)}]) \\[1mm]
            \hspace{0.6cm} (\rT_{k+1}, \, \mathrm{Mem}_{k+1})^\top\hspace{0.4cm}
            \leftarrow &\hspace{0.0cm} (\rT_{k} - \rC_k^{\mathrm{CT}} \!, \, \mathrm{Mem}_k - \rC^\mathrm{Mem}_k)^\top \\[1mm]
            \texttt{return } &\texttt{BMLSGD}(\bz_{L}^{(k+1)} \!, \, \eta \cdot \bnorm{E^{\text{ML}}[\bg_{L}^{(k)}]}_W) \\
        \end{cases} \\ \\[-1mm]
        &\texttt{function NotFeasible}(\tset{M_{k, \ell}^{\text{opt}}}_{\ell=0}^{L_k}) \colon \\[-1mm]
        &\quad
        \begin{cases}
            \texttt{if } \sum_{\ell=0}^{L_k} M_{k,\ell}^{\text{opt}} \widehat{\rC}^{\rC\rT}_{k - 1, \ell} > \rT_{k} \colon
            &\hspace{0.1cm} \texttt{return } \texttt{true} \\[1mm]
            \texttt{if } m_k h_{L_k}^{-\gamma_{\mathrm{Mem}}} > {\mathrm{Mem}}_{k} \colon
            &\hspace{0.1cm} \texttt{return } \texttt{true} \\[1mm]
            \texttt{else} \colon
            &\hspace{0.1cm} \texttt{return } \texttt{false}
        \end{cases} \\ \\[-1mm]
        &\texttt{function } \texttt{AdaptiveGradientDescent}(\bz_{L}^{(k)} \!, E^{\text{ML}} [\bq_{L}^{(k)}]) \colon \\[-1mm]
        &\quad
        \begin{cases}
            E^{\text{ML}} [\bg_{L}^{(k)}]
            &\leftarrow \hspace{0.3cm}
            \lambda \bz_{L}^{(k)} - E^{\text{ML}} [\bq_{L}^{(k)}] \\[2mm]
            \hspace{0.2cm} \widehat{c}_\mathrm{Lip}
            &\leftarrow \hspace{0.3cm}
            \frac{\tnorm{E^{\text{ML}}[\bg_{L}^{(k)}]-E^{\text{ML}}[\bg_{L}^{(k-1)}]}_{W}}
            {\tnorm{t_{k-1}E^{\text{ML}}[\bg_{L}^{(k-1)}]}_{W}}
            \hspace{0.4cm} \text{// } t_{k-1}, E^{\text{ML}}[\bg_{L}^{(k-1)}] \text{ from memory} \\[2mm]
            \hspace{0.2cm} t_{k}
            &\leftarrow \hspace{0.3cm}
            \frac{\tnorm{E^{\text{ML}}[\bg_{L}^{(k)}]}_{W}^2 - \widehat{\err}^{\text{sam}}_k}
            {\widehat{c}_{\mathrm{Lip}} \tnorm{E^{\text{ML}}[\bg_{L}^{(k)}]}_{W}^2}
            \hspace{1.15cm} \text{// } \widehat{\err}^{\text{sam}}_k
            \text{ estimated with~\eqref{eq:sampling-error-estimator}} \\[2mm]
            \hspace{0.2cm} \bz_{L}^{(k+1)}
            &\leftarrow \hspace{0.3cm}
            \pi_{Z} \big(\bz_{L}^{(k)} - t_{k} E^{\text{ML}} [\bg_{L}^{(k)}] \big) \\[2mm]
            \texttt{return } &\hspace{0.75cm} \bz^{(k+1)}, \,\, E^{\text{ML}} [\bg_{L}^{(k)}]
        \end{cases}
    \end{align*}
\end{algorithm}

    \subsection{Experiments}
\label{subsec:bmlsgd-experiments}

We now present our final numerical results,
evaluating the performance of \Cref{alg:bmlsgd}
on the PDE example from~\Cref{subsec:example-problem}.
In particular, we focus on estimating the convergence
rate~$\delta$ as defined in~\Cref{cor:upper-and-lower-bound}.
Following the approach of~\cite{baumgarten2024fully},
we estimate $\delta$ via linear regression on the
logarithmic decay of the gradient norm over time:
\begin{align*}
    {\argmin_{(\widehat{\delta}, \, \widehat{c}_{\delta})}}
    \sum_{k} \big( \log_2 \big(\bnorm{E^{\text{ML}}[\bg_L^{(k)}]}_{W} \big) - \widehat{\delta}
    \log_2(\rT_{k}) + \widehat{c}_{\delta} \big)^2
\end{align*}
All figures in this section display estimates of $\delta$.
To clarify the plots, we represent data points directly
(excluding early iterations within the first 60 seconds)
along with the corresponding linear fit.

\paragraph{Method comparison}
\Cref{fig:metod-comparison} presents a direct comparison
between the BMLSGD and BSGD methods.
To improve the BSGD performance,
we used a decaying step size $t_k = t_0 \cdot K^{-0.5}$
satisfying~\eqref{eq:step-size-assumption} with $t_0=250$
and increased the iteration count to $K = 150$.
With $h_\ell = 2^{-7}$ and $M_\ell = 256$,
this configuration gave the best BSGD results
within one hour of computation.
Despite these optimizations,
BMLSGD significantly outperforms BSGD on our example — achieving
comparable results 18× faster and yielding errors 5× smaller
for the same computational cost.
As shown in \Cref{fig:metod-comparison},
BMLSGD also achieves a markedly better convergence rate
of $\delta \approx 0.5$,
compared to $\delta \approx 0.37$ for BSGD.
This aligns with theoretical expectations:
for BSGD, $\delta = \tfrac{\alpha}{2\alpha+\gamma}$
(cf.~discussion in Sections~\ref{subsec:bsgd-discussion} and~\ref{sec:multilevel-monte-carlo}),
while for BMLSGD, \Cref{cor:upper-and-lower-bound} gives
$\delta = \min \tset{\tfrac{1}{2}, \tfrac{\alpha}{2\alpha + (\gamma - \beta)}}$.
Finally, BMLSGD achieves these improvements with fewer iterations,
lower bias, and reduced noise (see left plot of~\Cref{fig:metod-comparison}).

\begin{figure}
    \begin{center}
        \includegraphics[width=0.8\textwidth]{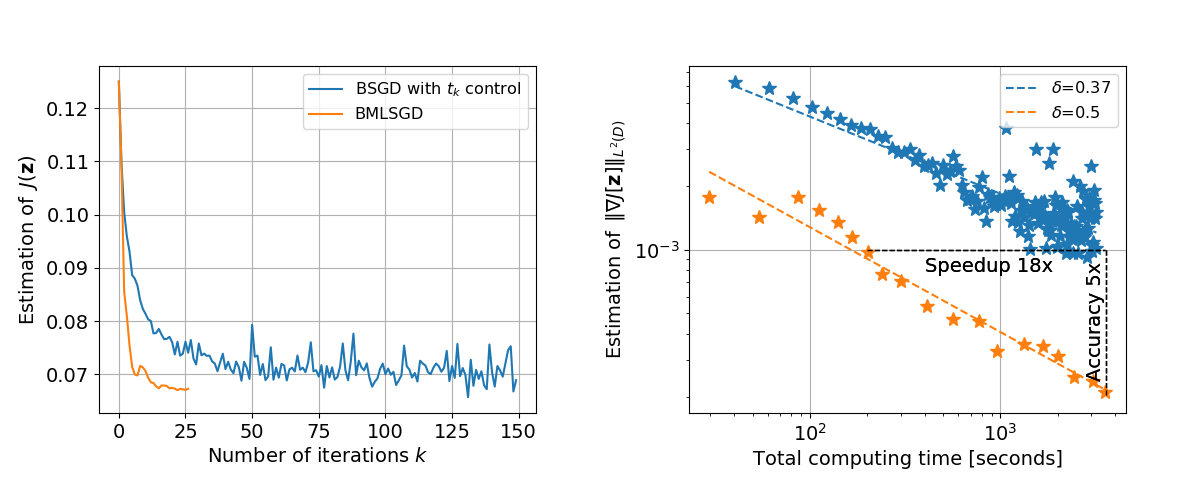}
    \end{center}

    \vspace{-0.6cm}

    \caption{Comparison of \Cref{alg:bsgd} with \Cref{alg:bmlsgd}.}
    \label{fig:metod-comparison}
\end{figure}

Unlike in Sections~\ref{subsec:bsgd-experiments} and~\ref{subsec:mlsgd-experiments},
as well as the preceding discussion, Figures~\ref{fig:adaptive} and~\ref{fig:nodes}
additionally present results on multilevel statistics and highlight the decay
of the gradient $\tnorm{E^{\text{ML}}[\bg_L^{(k)}]}_{W}$ over time.
The aim is to justify~\Cref{assumption:mlmc} by estimating the exponents
$\alpha_{\bu}, \alpha_{\bq}, \beta_{\bv}, \beta_{\bp}, \gamma_{\rC\rT}$, and $\gamma_{\mathrm{Mem}}$.
The top row in both figures displays the experimentally measured exponents:
for the state (dash-dotted line) and the adjoint solutions (dashed line)
in the left and center plots, respectively.
The rightmost plot shows the measured computational cost exponents $\gamma^{\mathrm{CT}}$ and $\gamma^{\mathrm{Mem}}$
in a dual-axis format:
the increasing solid line indicates average computational costs per level,
while the bar plot represents memory costs per level.
The horizontal lines in this plot mark the total memory footprint in megabytes.
The slight overhead of memory costs at the lowest level stems from initialization
of the algorithm and the solvers.

\paragraph{Step size choice}
Following the approach of~\cite{baumgarten2024fully},
we assess the multilevel results to identify the optimal step size rule
under a fixed computational budget through direct performance comparisons.
The algorithm is initialized with the multilevel batch configuration:
$(h_0, M_0)^\top = (2^{-4}, 64)^\top$, $(h_1, M_1)^\top = (2^{-5}, 16)^\top$,
and $(h_2, M_2)^\top = (2^{-6}, 4)^\top$.
Performance is evaluated using the same settings as in~\Cref{fig:metod-comparison}.
In~\Cref{fig:adaptive}, we compare fixed step sizes $t_k \equiv 100$ and $t_k \equiv 150$
with the adaptive step size rule~\eqref{eq:adaptive-step-size} using $t_0 = 200$.
Note that the method does not necessarily require~\eqref{eq:step-size-assumption}
to be fulfilled as long as the constant step size keeps $\rho_k$ below one in~\eqref{eq:quadratic-minimization}
The center plot in the second row of~\Cref{fig:adaptive} shows
that all three configurations result in a comparable load distribution across levels
and remain within the 1-hour computational budget
(indicated by the red line above the horizontal lines representing the total time cost of each run).
The lower right plot indicates that the adaptive step size yields the best
convergence rate and lowest gradient norm at the end of the optimization.
In the context of BMLSGD, the adaptive rule utilizes each batch optimally,
thereby enabling larger batches that lead to improved parallel efficiency.

\paragraph{Node scaling}
Finally, we demonstrate that the BMLSGD method scales well with
increased computational resources (cf.~\Cref{fig:nodes}).
To this end, we run the method with adaptive step size and $t_0 = 200$
on $\abs{\cP} = 64$ (1 node), $\abs{\cP} = 256$ (4 nodes), and $\abs{\cP} = 1024$ (16 nodes).
Notably, the method leverages the additional resources to compute more
samples—and, for $\abs{\cP} = 1024$, also includes an additional
level (cf.~lower left plot of~\Cref{fig:nodes} showing the total
number of samples $M_\ell$ over the optimization).
As shown in the lower right plot of~\Cref{fig:nodes},
and similarly observed in~\cite{baumgarten2024fully},
increasing computational resources improves the solution quality.
The smaller gap between the green ($\abs{\cP} = 1024$)
and orange ($\abs{\cP} = 256$) lines,
compared to that between the orange and blue ($\abs{\cP} = 64$),
suggests diminishing parallel efficiency with more nodes.
This is due to the inherently unparallelizable portion $\lambda_{\rp}$
of the code (see~\Cref{cor:upper-and-lower-bound} and~\cite{baumgarten2024fully} for further details).

\begin{figure}
    \hspace{-8mm}
    \includegraphics[width=1.05\textwidth]{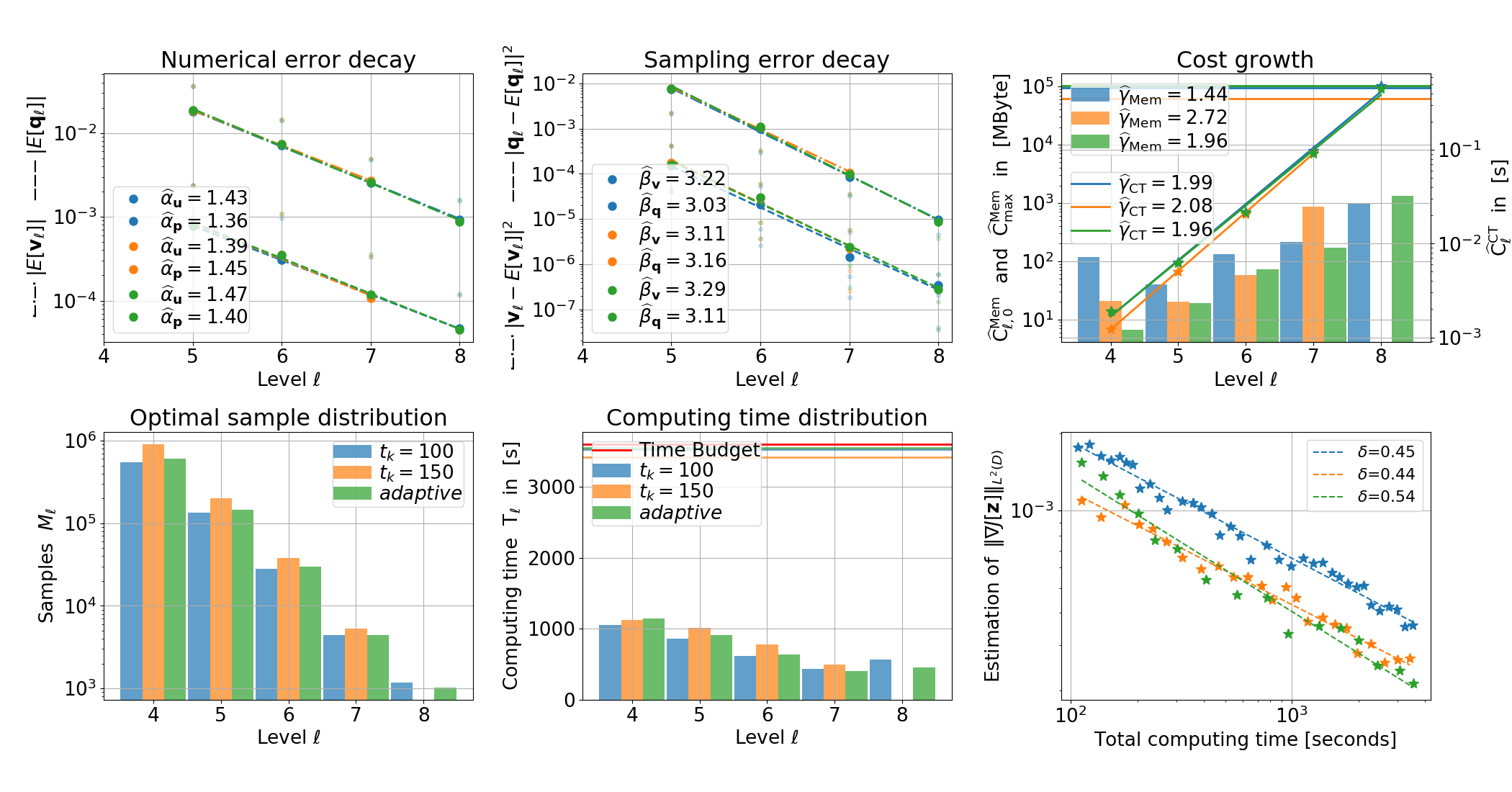}
    \vspace*{-1.0cm}
    \caption{Comparison of $t_k\equiv 100$ and $t_k\equiv 150$
        with the adaptive step size rule~\eqref{eq:adaptive-step-size}.}
    \label{fig:adaptive}
\end{figure}
\begin{figure}
    \hspace{-8mm}
    \includegraphics[width=1.05\textwidth]{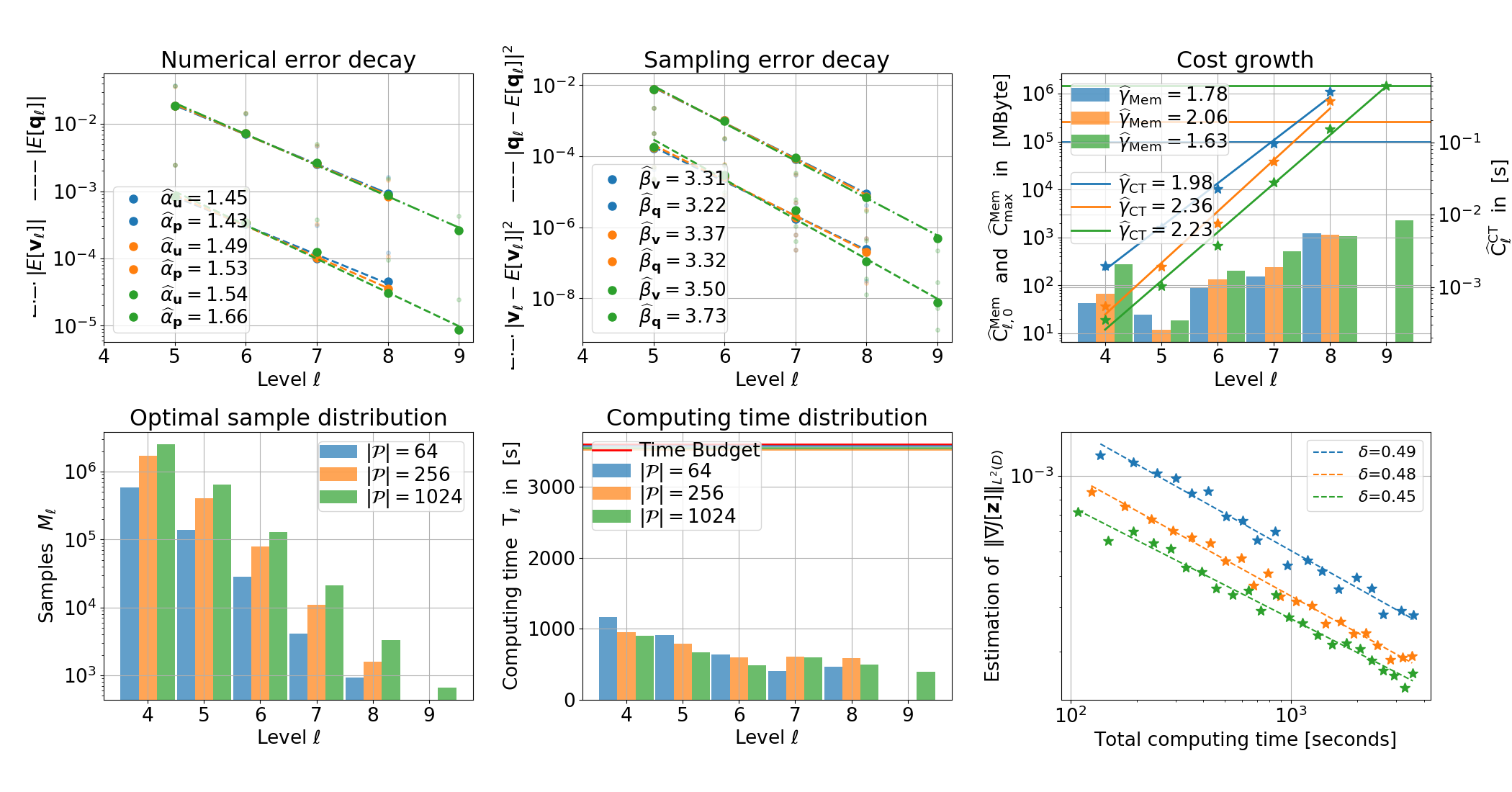}
    \vspace*{-1.0cm}
    \caption{CPU-scaling experiment with $\abs{\cP}=64$, $\abs{\cP}=256$ and $\abs{\cP}=1024$.}
    \label{fig:nodes}
\end{figure}

    \section{Outlook and Conclusion}\label{sec:discussion-and-outlook}

In this paper, we introduced the MLSGD method and its budgeted variant.
We showed that the method solves optimal control problems of the
form~\eqref{problem:ocp} with a linear convergence rate in the number of optimization steps.
The total error is controlled via multilevel Monte Carlo estimation,
as established in \Cref{lem:error-gradient-estimation},
\Cref{thm:convergence-mlsgd}, and \Cref{cor:convergence-mlsgd}.
Our numerical experiments in~\Cref{subsec:mlsgd-experiments}
and~\Cref{subsec:bmlsgd-experiments} clearly demonstrate the superior
performance of MLSGD compared to the baseline BSGD method.
In particular, the budgeted variant of~\Cref{sec:adaptivity-and-budgeting}
shows significant improvements in convergence speed, accuracy, parallel scalability,
and robustness, while ensuring total error control relative to available CPU-time and memory,
as shown in~\Cref{cor:upper-and-lower-bound}.


Future work may explore other PDEs, such as wave equations~\cite{baumgarten2024fully}
or coupled PDEs~\cite{baumgarten2023fully}, applications which are already
supported by our software M++~\cite{baumgarten2021parallel}.
The method is designed for high-dimensional problems and large computing clusters,
where its inherent parallelism, based on~\cite{baumgarten2025budgeted},
enables efficient usage of resources.
Algorithmically, the method can be extended in several directions,
such as incorporating limited-memory BFGS methods~\cite{liu1989limited}
as an alternative to gradient descent,
or combining it with QMC methods as in~\cite{guth2021quasi, guth2024parabolic}.
Further improvements may also be achieved by refining the step size strategies
and by using new averaging schemes~\cite{kohne2025exponential}.
We also suspect that the MLSGD method can excel in
risk-averse PDE-constrained optimization using conditional value-at-risk (CVaR) estimation,
as described in~\cite{kouri2016risk}, potentially leveraging the multilevel techniques
introduced in~\cite{ayoul2023quantifying} to estimate the CVaR.
Lastly, we want to point out again that the MLSGD method
is not limited to PDE-constrained optimization in application.
We note that simultaneously with our work,
a related formulation to the MLSGD method appeared in
\cite{rowbottom2025multi}, focusing on operator learning.
We believe that the findings in our paper,
e.g.~how to realize adaptivity and parallel scalability,
will improve future applications of the MLSGD method
in machine learning.

    \section*{Acknowledgments}

We thank Christian Wieners, Sebastian Krumscheid and Robert Scheichl
for their feedback and advice on the presented work.
We further acknowledge the technical support by the National
High-Performance Computing Center (NHR) at KIT and the possibility
to access to the HoreKa supercomputer.
LLMs were used to correct some spelling and grammar of this article.
Funded by the Deutsche Forschungsgemeinschaft (DFG, German Research Foundation) – Project-ID 258734477 – SFB 1173.

    \bibliographystyle{ieeetr}

    \bibliography{bibliography}

\end{document}